\documentclass[11pt, a4]{article} 
\usepackage{times}
\usepackage{amsmath,amsthm,amsfonts, amssymb, amstext}
\usepackage[margin=2.5cm]{geometry}
\usepackage{enumitem}
\usepackage{soul}

 \usepackage{authblk}    
 \usepackage[backend=bibtex]{biblatex}
\usepackage[skip=8pt]{parskip}

\usepackage[dvips]{graphicx} 
\usepackage{mathrsfs}
\usepackage{hyperref}
\usepackage{tocloft}                
\usepackage[dvipsnames]{xcolor}
       
\usepackage{marvosym}%\Smiley{} \Frowny{}
 
 \usepackage{constants}         
 \newconstantfamily{Ctilde}{symbol=\widetilde C}
 
\newcommand{\foot}[1]{\mbox{}\marginpar{\raggedleft\hspace{0pt}\tiny #1}}
\newcommand{\footsl}[1]{\foot{SL:  #1}}

\makeatother
\newlength\mylen

\renewcommand\cftpartpresnum{Part~}
 \usepackage{hyperref}
\usepackage{varioref}  
\usepackage{wrapfig}                                                 
\usepackage{color}
\usepackage{psfrag}

\newcommand{\todelete}[1]{{\color{red}#1}} 
\newcommand{\toadd}[1]{{\color{blue}#1}} 
\newcommand{\comment}[1]{{\color{OliveGreen}#1}}   
\newcommand{\nocomment}[1]{{\color{red}#1}}  

 \newcommand{\quand}{\quad\text{ and } \quad}
\theoremstyle{plain}

\newtheorem{theorem}{Theorem}[section]
\newtheorem{proposition}[theorem]{Proposition}
\newtheorem{lemma}[theorem]{Lemma}
\newtheorem{sublemma}[theorem]{Sublemma}
\newtheorem{corollary}[theorem]{Corollary}

\newtheorem*{theorem*}{Theorem}
\newtheorem*{main theorem}{Main Theorem}
\newtheorem*{corollary*}{Corollary}
\newtheorem*{conjecture*}{Conjecture}

\theoremstyle{definition}
\newtheorem{definition}[theorem]{Definition}
\newtheorem*{definition*}{Definition}

\theoremstyle{remark}
\newtheorem{remark}{Remark}[section]
\newtheorem*{remark*}{Remark}
\newtheorem*{remarks*}{Remarks}

\newcommand{\nc}{\newcommand}

\nc{\inmath}{\ensuremath{\in}}

\nc{\abs}[1]{\ensuremath{|#1|}}
\nc{\stars}{\ensuremath{ (*)_{n}^{s} } }
\nc{\starc}{\ensuremath{ (*)_{n}^{c} } }
\nc{\ang}{\sphericalangle}

\renewcommand{\cal}[2]{\ensuremath{\mathcal{#1}^{(#2)}}}
\nc{\tcal}[2]{\ensuremath{\mathcal{\tilde{#1}}^{(#2)}}}
\nc{\ttcal}[2]{\ensuremath{\mathcal{\tilde\tilde{#1}}^{(#2)}}}
\nc{\hcal}[2]{\ensuremath{\mathcal{\widehat{#1}}^{(#2)}}}

\nc{\EG}[1]{\ensuremath{\text{(EG)}_{#1}}}
\nc{\BD}[1]{\ensuremath{\text{(BD)}_{#1}}}
\nc{\HV}[1]{\ensuremath{\text{(HV)}_{#1}}}
\nc{\BoD}[1]{\ensuremath{\text{(BoD)}_{#1}}}
\nc{\AD}[1]{\ensuremath{\text{(AD)}_{#1}}}
\nc{\BDT}[1]{\ensuremath{\text{(BDT)}_{#1}}}
\nc{\BP}[1]{\ensuremath{\text{(BP)}_{#1}}}
\nc{\GC}[1]{\ensuremath{\text{(GC)}_{#1}}}

\nc{\negfour}{\!\!\!\!}
\nc{\negsix}{\!\!\!\!\!\!}

\nc{\D}[1][]{\ensuremath{\Delta^{#1}}}
\nc{\Dp}[1][]{\ensuremath{\Delta_{+}^{#1}}}

\nc{\eg}{\ensuremath{\varepsilon^{\gamma}}}
\nc{\egd}{\ensuremath{\varepsilon^{\gamma+\delta}}}
\nc{\egdi}{\ensuremath{\varepsilon^{\gamma+\delta+\iota}}}
\nc{\eggdd}{\ensuremath{\varepsilon^{2(\gamma+\delta)}}}
\nc{\egddi}{\ensuremath{\varepsilon^{\gamma+2\delta+\iota}}}
\nc{\eggddi}{\ensuremath{\varepsilon^{2\gamma+2\delta+\iota}}}
\nc{\el}{\ensuremath{e^{\frac{1}{\lambda}}}}
\nc{\ea}[1]{\ensuremath{e^{-\alpha #1}}}
\nc{\eb}[1]{\ensuremath{e^{-2 \alpha #1}}}
\nc{\ega}[1]{\ensuremath{\eg e^{-\alpha #1}}}
\nc{\egb}[1]{\ensuremath{\eg e^{-2\alpha #1}}}

\nc{\gd}{\ensuremath{{\gamma+\delta}}}
\nc{\gdi}{\ensuremath{{\gamma+\delta+\iota}}}
\nc{\gddi}{\ensuremath{{\gamma+2\delta+\iota}}}
\nc{\ggddi}{\ensuremath{{2\gamma+2\delta+\iota}}}

\nc{\be}{\ensuremath{b\varepsilon_{*}}}

\nc{\df}[1][x]{\varphi'(#1)}
\nc{\adf}[1][x]{|\varphi'(#1)|}
             
\nc{\pfi}[2]{\ensuremath{\partial_{#1}\Phi_{#2}}}
\nc{\pd}[2]{\ensuremath{\partial_{#1}^{#2}}}
\nc{\pdtil}[2]{\ensuremath{\tilde\partial_{#1}^{#2}}}

\nc{\F}[2]{\ensuremath{\Phi^{#1}(#2)}}
\nc{\Fxi}[1]{\Phi(\xi_{#1})}
\nc{\Fxit}[1]{\Phi(\tilde\xi_{#1})}

\nc{\DFz}[1][]{\ensuremath{D\Phi (z_{#1})}}
\nc{\DFzv}[2][]{\ensuremath{D\Phi (z_{#1})\cdot{#2}}}
\nc{\DFxi}[1]{\ensuremath{D\Phi (\xi_{#1})}}
\nc{\DFxiv}[2]{\ensuremath{D\Phi 
                                      (\xi_{#1})\cdot{#2}}}
\nc{\DFxit}[1]{\ensuremath{D\Phi (\tilde{\xi}_{#1})}}
\nc{\DFxitv}[2]{\ensuremath{D\Phi({\tilde\xi}_{#1})\cdot{#2}}}

\nc{\w}[1]{\ensuremath{w_{#1}}}
\nc{\wt}[1]{\ensuremath{{\tilde{w}}_{#1}}}
\nc{\om}[1]{\ensuremath{\omega_{#1}}}
\nc{\omt}[1]{\ensuremath{{\tilde{\omega}}_{#1}}}
\nc{\sig}[1]{\ensuremath{\sigma_{#1}}}
\nc{\sigt}[1]{\ensuremath{{\tilde{\sigma}}_{#1}}}

\nc{\vect}[1]{\ensuremath{
                      ({\tilde{\omega}}_{#1}, {\tilde{\sigma}}_{#1}) }}

\nc{\thetasub}[1]{\ensuremath{\theta_{#1}}}
\nc{\thetatsub}[1]{\ensuremath{\theta_{#1}}}

\nc{\z}[1]{\ensuremath{z_{#1}}}
\nc{\xisub}[1]{\ensuremath{\xi_{#1}}}
\nc{\xit}[1]{\ensuremath{\tilde{\xi}_{#1}}}
\nc{\xitsub}[1]{\ensuremath{\tilde\xi_{#1}}}
\nc{\g}[1]{\ensuremath{\gamma_{#1}}}
\nc{\gt}[1]{\ensuremath{\tilde{\gamma}_{#1}}}

\nc{\coni}[1]{\ensuremath{e^{(#1)} }}
\nc{\exi}[1]{\ensuremath{f^{(#1)} }}
\nc{\con}[2]{\ensuremath{e^{(#1)}(#2) }}
\nc{\ex}[2]{\ensuremath{f^{(#1)}(#2) }}
\nc{\conim}[3]{\ensuremath{e^{(#1)}_{#2}(#3) }}
\nc{\exim}[3]{\ensuremath{f^{(#1)}_{#2}(#3) }}
\nc{\conimi}[2]{\ensuremath{e^{(#1)}_{#2} }}
\nc{\eximi}[2]{\ensuremath{f^{(#1)}_{#2} }}

\nc{\crit}[1]{\ensuremath{z^{(#1)}}}
\nc{\critim}[2]{\ensuremath{z^{(#1)}_{#2}}}
\nc{\critx}[1]{\ensuremath{x^{(#1)}}}
\nc{\critimx}[2]{\ensuremath{x^{(#1)}_{#2}}}
\nc{\zetacrit}[1]{\ensuremath{\zeta^{(#1)}}}
\nc{\zetacritim}[2]{\ensuremath{\zeta^{(#1)}_{#2}}}
\nc{\critset}[1]{\ensuremath{\mathcal{C}^{(#1)}}}
\nc{\critsetim}[2]{\ensuremath{\mathcal{C}^{(#1)}_{#2}}}

\begin{document}
\title{
Finite Time Hyperbolic Coordinates}
\date{}

\author{Stefano Luzzatto, Dominic Veconi, Khadim War}

%\affil[1]{Abdus Salam International Centre for Theoretical Physics (ICTP). Trieste, Italy, }
%\affil[2]{Wake Forest University, Winston-Salem, NC, USA}
\maketitle

\begin{abstract}
We define \emph{finite-time hyperbolic coordinates}, describe their geometry, and prove various results on both their convergence as the time scale increases, and on their variation in the state space. Hyperbolic coordinates reframe the classical paradigm of hyperbolicity: rather than define a hyperbolic dynamical system in terms of a splitting of the tangent space into stable and unstable subspaces, we define hyperbolicity in terms of the \emph{co-eccentricity} of the map. The co-eccentricity describes the distortion of unit circles in the tangent space under the differential of the map. Finite-time hyperbolic coordinates have been used to demonstrate the existence of SRB measures for the H\'enon map; our eventual goal is to both elucidate these techniques and to extend them to a broad class of nonuniformly and singular hyperbolic systems. 
%We introduce the notion of \emph{(finite time) hyperbolic coordinates}, prove some of their properties and discuss their applications. 
%\comment{SL: I think this abstract is excellent, just maybe a little bit too long. Maybe delete the sentences in red?}

 \end{abstract}	
	
\setcounter{tocdepth}{2}
% 	\enlargethispage{2cm}

%\tableofcontents
%\newpage
%\part{Introduction}
  \label{intro}

%\newpage

\section{Introduction}

\subsection{Physical Measures}
One of the most interesting and challenging problems in the theory of dynamical systems is that of describing the \emph{statistical} properties  of a map \( f: M \to M \) on some (typically compact) metric space $M$. Given an initial condition \( x\in M \), for every \( n \geq 1 \), we can define the \emph{empirical measure} 
\[
e_{n}(x) := \frac{1}{n} \sum_{k=0}^{n-1}\delta_{f^{i}(x)}
\]
associated to the orbit of \( x \). Notice that \( e_{n}(x) \) is an atomic probability measure uniformly distributed on the first \( n \) points of the orbit of  \( x \). The measure $e_n$ describes the frequency with which the orbit visits various regions in  \( M \). If the sequence of empirical measures \emph{converges}  in the weak topology to a probability measure $\mu$ as $n \to \infty$,  then \( \mu\) can be thought of as describing the \emph{asymptotic statistical  distribution} of the orbit. In this situation, we say that the point \( x \) has statistical behavior described by the probability distribution $\mu$. We define the \emph{basin} of \( \mu \) as the set $B_\mu$ of points whose asymptotic statistical distribution is described by the measure \( \mu\). Formally,
\[
 \mathcal B_{\mu}:=\{x \in M: e_{n}(x) \to \mu\}.
 \]
  If \( M \)  is endowed with a normalized reference probability measure \( m \) (often a Lebesgue measure or Riemannian area form), we say that \( \mu \) is a \emph{physical measure} if 
  \[ m( \mathcal B_{\mu})>0
  \] 
  since this means that there is a ``physically observable'' set of points whose asymptotic statistical distribution is described by the given measure \( \mu \). If there is some finite set of probability measures such that the union of the corresponding basins has full Lebesgue measure, then we can say in some sense that we have succeeded in describing the dynamics from a probabilistic/statistical point of view. A major and ongoing line of research over the last several decadesis establishing the existence and uniqueness of physical measures in specific classes of smooth dynamical systems. 

Before we discuss in the forthcoming section some technqiues for constructing physical measures, we first remark that the existence of physical measures is not guaranteed: there are dynamical systems that admit no physical measures at all. The simplest example is the identity map \( f(x) = x \), for which \( e_{n}(x) = \delta_{x}\)  for every \( x \) and  \( n \geq 1 \), and therefore trivially \( e_{n}(x) \to \delta_{x}\). It follows that there cannot be any physical measures, since $\mathcal B_{\delta_x} = \{x\}$, and therefore $m(B_{\delta_x}) = 0$ for any Lebesgue reference measure $m$ since a Lebesgue reference measure must be nonatomic. More interesting examples are  maps for which  the empirical measures \( e_{n}(x)\) \emph{do not converge}. This can be very counter-intuitive since if there are two distinct probability measures \( \mu, \nu\) and subsequences \( n_{i}, n_{j}\to \infty\) such that \( e_{n_{i}}(x) \to \mu\) and \( e_{n_{j}}(x)\to \nu\), this means that the statistical distribution of the orbit  \emph{depends on the time scale}. In a rough sense, this means that the statistical behavior of the map is sometimes described by the measure \( \mu \) and sometimes  by the measure \( \nu\).  This  means  that there is no well defined asymptotic statistical distribution, in which case we say that the orbit of \( x \) exhibits \emph{non-statistical}, or  \emph{historic}, behaviour, see~\cite{HofKel90, KirSom17, LabRod17, Tak08,  Tal22}. The  \emph{Palis Conjecture} \cite{Pal00} says that non-statistical dynamical systems are ``exceptional,'' whereas  ``typical'' dynamical systems admit a finite number of physical measures whose unions of basins have full Lebesgue measure. 

\subsection{Hyperbolicity}
The simplest examples of physical measures are Dirac-delta measures on attracting periodic orbits. However, there are huge classes of systems which do not have attracting periodic orbits and for which proving the existence of physical measures is highly challenging. 
Over the last 50 years, starting mostly from the work of Sinai, Bowen, and Ruelle in the 1970s \cite{Bow75, BowRue75, Sin72},  there has been a huge amount of research on developing techniques to prove the existence of physical measures. Most of these techniques assume some kind of  \emph{hyperbolicity}, which essentially consists of a \(Df\)-invariant  \emph{tangent bundle decomposition} 
\begin{equation}\label{eq:hypsplit}
T_{x}M = E^{1}_{x}\oplus \cdots \oplus E^{\kappa}_{x}
\end{equation}
satisfying various properties. These properties include  estimates on the contraction or expansion of vectors in each sub-bundle, relation between the contraction and expansion in different sub-bundles, and  the regularity of the decomposition (i.e., whether the distributions $E^j_x$ depend measurably or continuously on \( x \)). See \cite{BarPes07} for a comprehensive survey. While there is no completely general result  which says that every form of hyperbolicity implies the existence of a physical measure, there are many highly non-trivial results which show that even some very weak hyperbolicity conditions can be sufficient to prove the existence of a physical measure \cite{CliLuzPes23}.

\subsection{Verifying Hyperbolicity}
Notwithstanding the importance of the results mentioned above, using hyperbolicity as an \emph{assumption} naturally leads to the question of whether it can be \emph{verified} in specific systems. 
As it turns out, there are many situations in which hyperbolicity can be verified relatively easily (at least in principle), either through geometric and analytic arguments \cite{PalTak93} and/or explicit rigorous numerical computations \cite{Ara07}. Most of these easily verifiable situations occur in the context of\emph{``uniformly hyperbolic''} systems,  in which the hyperbolic splitting is continuous, robust under small perturbations, and the contraction and expansion are uniform in  the phase space. 

However, there are more general situations with broader applications than uniform hyperbolicity, in which the hyperbolic splitting is only measurable, the splitting can be easily destroyed under small perturbations, and the contraction and expansion estimates cannot be given uniformly in the phase space. Here we loosely refer to these dynamical systems as ``nonuniformly hyperbolic.'' The verification of the existence of physical measures in nonuniformly hyperbolic systems is still extremely hard. 
A major breakthrough was made in 1991 by Benedicks and Carleson in their famous paper \cite{BenCar91},  in which they studied the H\'enon family \( H_{a,b}(x,y) = (1+y-ax^{2}, bx)\) of two-dimensional diffeomorphisms of \( \mathbb R^{2}\). They developed a number of very delicate and sophisticated arguments and estimates, which we will for simplicity refer to as \emph{parameter-exclusion} arguments, to prove that when the  map is sufficiently \emph{strongly dissipative} (i.e. when the parameter \( b \) is sufficiently small), there exists a positive Lebesgue measure set of parameters \( a \) for which the map \( H_{a,b}\) admits some (nonuniformly) hyperbolic structure. It was then shown in \cite{BenYou93} that this hyperbolic structure implies the existence of a physical measure. 

There exist several extensions and generalizations of the arguments and results of Benedicks and Carleson \cite{BenCar91}, but they typically deal with families of maps that share a lot of the key features of the H\'enon maps, such as \emph{smoothness} and \emph{strong dissipativity}.
There are many classes of nonuniformly hyperbolic dynamical systems that do not share these features, such as the famous and very well-studied \emph{Standard Map} \cite{BloLuz09, GioLaz00}, which is area-preserving and therefore not dissipative. A family of dissipative examples of nonuniformly hyperbolic systems come from the \emph{Lorenz equations} \cite{Lor63},  whose two-dimensional Poincar\'e ``Lorenz'' map lacks smoothness: the Poincar\'e map contains a non-trivial discontinuity curve, giving rise to a curve of singularities with unbounded derivatives. For certain parameter values, non-rigorous numerical studies strongly suggest the existence of some kind of hyperbolic structure, but the rigorous verification of this  remains an open problem and the existing techniques do not seem to easily generalize to these systems.

%Notwithstanding many attempts, there have not  been any alternative approaches which have had any real success in the concrete verification of (nonuniform) hyperbolicity and therefore a natural question is whether the techniques of Benedicks-Carleson can be generalized to apply to such more general settings. The original arguments, as well as their generalizations, use smoothness and strong dissipativity in a number of places, and it is far from obvious how these can be relaxed. A major challenge comes from the fact that there are a number of arguments and computations which are very technical and very strongly intertwined, with estimates in one argument depending very delicately on the estimates and constants arising from other arguments in ways which are hard to fully pinpoint. This has made the whole approach very hermetic and difficult to understand to large sections of the community of researchers in the area, as well as hampering the further development of the techniques introduced in \cite{BenCar91}. 

\subsection{Goals of this paper}
 The original  pioneering parameter-exclusion arguments  of Benedicks and Carleson, as well as those developed in its generalizations,  are \emph{extremely intricate}, consisting of many ``sub-arguments'' that are tightly inter-woven. We believe that this is the main reason for which only a handful of researchers have taken the time to properly understand the methodology of the proof, and the reason for which there have been no significant generalizations to this methodology in the literature outside of strongly dissipative smooth settings.
  
This is the first in a series of  papers whose goal is to  gradually \emph{disentangle} the various  strands of the existing parameter-exclusion arguments. Our hope is to not only make the original argument much more accessible, but also highlight that many of the sub-arguments in the construction  are of independent interest and can be generalized to maps which are not necessarily strongly dissipative and may have discontinuities or singularities. Ultimately this should make it possible to develop generalizations of the entire parameter-exclusion arguments to systems such as the Standard map or Lorenz-like maps.

We begin this program with a careful study of a comparatively elementary notion, referred to in  \cite{BenCar91} as the \emph{most contracted} direction of the differential map. This is somewhat understated in the original arguments, but turns out to play a crucial role in constructing some geometric ``hyperbolic'' structures that depend on only \emph{finitely} many iterates of the map. By comparison, the usual notion of hyperbolicity is essentially asymptotic and therefore requires information about all iterates. Various properties of these finite-time most contracted directions, such as their dependence on the iterations and the base point, constitute important conditions for the development of the overall argument. These properties seem to rely significantly on the smoothness and especially the strong dissipativity of the maps. 
 
In this paper we will give a formal definition of what we call  \emph{(finite time) Hyperbolic Coordinates}, which we believe is the natural and more intuitive setting in which to understand and study the  ``most contracted directions'' of \cite{BenCar91}. We will introduce a general and quite  weak \emph{pointwise} and \emph{finite-time}  hyperbolicity condition, which we call \emph{quasi-hyperbolicity}, and show that under this assumption the hyperbolic coordinates satisfy a number of important properties. These properties, in particular, include those required for the parameter-exclusion arguments of \cite{BenCar91} and its generalizations. Crucially, however, our quasi-hyperbolicity condition \emph{does not require the map to be strongly dissipative}: it can apply even to area-preserving maps such as the Standard map. Moreover, \emph{nor does our condition require the map to have a bounded derivative}, thus making it applicable to systems with singularities such as Lorenz-like maps. This broad applicability makes the results completely new as they are not included, even implicitly, in any of the existing literature as far as the authors are aware.

\section{Definitions and Statements of Results}

 In Section \ref{sec:hypcoord} we give the definition of finite-time \emph{hyperbolic coordinates} and make several remarks about the motivation for this notion. Then in Section \ref{sec:quasihyp} we introduce our condition of pointwise and finite time \emph{quasi-hyperbolicity}, and in Section \ref{sec:results} we state our main results.

\subsection{Hyperbolic Coordinates}\label{sec:hypcoord}
In this section we give the formal definition of Hyperbolic Coordinates and discuss some of their properties. 
We assume that  \( M \) is  a Riemann surface and that \( \Phi: M \to M \) is a  map. Our results are   \emph{pointwise} in the sense that they apply to individual orbits, so we do not assume any global regularity of \( \Phi \). Instead, we fix some 
 \( \xi_{0}\in M \) and some  \( k\geq 1 \) and suppose \( \Phi^{k} \)  is \( C^2\) at \( \xi_{0}\) (for the definition we only need \( C^{1}\) but for many properties and for our results we will need \( C^{2}\)).

\subsubsection{Definition of Co-eccentricity and Hyperbolic Coordinates}

 We define the  \emph{co-eccentricity} of 
\( \Phi^{k}_{\xi_{0}}\) or, more precisely, of the derivative map   
\( D\Phi^k_{\xi_0}: T_{\xi_0}M\to T_{\xi_k}M\), as 
\begin{equation}\label{eq:eccentricity1}
C_{\xi_{0}, k}  :=\frac{|\det(D\Phi^{k}_{\xi_{0}})|}{\|D\Phi_{\xi_{0}}^{k}\|^{2}}= 
\frac{\|(D\Phi_{\xi_{0}}^{k})^{-1}\|^{-2}}{|\det(D\Phi^{k}_{\xi_{0}})|}
= \frac{\|(D\Phi_{\xi_{0}}^{k})^{-1}\|^{-1}}{\|D\Phi_{\xi_{0}}^{k}\|}.
\end{equation}
Notice that \( \|(D\Phi_{\xi_{0}}^{k})^{-1}\|^{-1}\) is a somewhat convoluted way of writing the \emph{co-norm} of \( D\Phi_{\xi_{0}}^{k}\), that is, the norm of the image of the most contracted unit vector. The equality between the three expressions in \eqref{eq:eccentricity1} then  follows immediately from the fact that \( \det(D\Phi^k_{\xi_0}) =  \|D\Phi_{\xi_{0}}^{k}\| \|(D\Phi_{\xi_{0}}^{k})^{-1}\|^{-1}  \). The third formulation in \eqref{eq:eccentricity1} clearly shows that we  always have  \(  C_{\xi_{0}, k} \leq 1 \). Letting \( \mathcal S_0\subset T_{\xi_0}M \) be the unit circle and  \(\mathcal S_k := D\Phi_{\xi_0}^k(\mathcal S_0) \subset T_{\xi_k}M \) be its image,  the co-eccentricity  has a very natural geometrical interpretation:  if  \(  C_{\xi_{0}, k} = 1 \), then \( \mathcal S_k\) is also the unit circle, whereas if  \(  C_{\xi_{0}, k} < 1 \), then \( \mathcal S_k\) is a non-trivial ellipse and  there  are distinct unit vectors  \( e^{(k)}, f^{(k)}\in T_{\xi_0}M\) that map to  the minor and major semi-axes of the ellipse \( \mathcal S_k\), and are therefore respectively the \emph{most contracted} and \emph{most expanded} unit vectors for \( D \Phi^k_{\xi_0}\). 

\begin{definition}
\label{def:hypcoords} 
If $C_{\xi_0, k} < 1$, the  coordinates 
\[
\mathcal H^{(k)}= \{ e^{(k)}, f^{(k)}\}
\]
  defined by taking \( e^{(k)}, f^{(k)}\)  as unit basis vectors, are called  \emph{hyperbolic coordinates} of order $k$ at \( x_0\).
\end{definition}

Notice that  $e^{(k)}$ is  the \emph{most contracted}  unit vector and $f^{(k)}$ as the \emph{most expanded} unit vector under $D\Phi_{\xi_0}^k$, but these are just \emph{relative} terms and these vectors may not actually be expanded or contracted at all. 
Notice also that hyperbolic coordinates are not uniquely defined since \( -e^{(k)}\) and \( -f^{(k)}\) are also most contracted and most expanded respectively. We therefore just assume that some choice has been made and, as we shall see, this will not create any ambiguity or confusion in the settings which we will consider. 

\subsubsection{Hyperbolic Coordinates as Diagonalizing Coordinates}
Hyperbolic coordinates are  useful in a number of ways.  First of all, note that they form an \emph{orthonormal} basis of $T_{\xi_0}M$: if $\mathcal S_0 \subset T_{\xi_0}M$ is the unit circle and $\mathcal S_k = D\Phi_{\xi_0}^k(\mathcal S_0) \subset T_{\xi_k}M$ is the ellipse given by the image of $\mathcal S_0$ under $D\Phi_{\xi_0}^k$, then it is a fundamental result in linear algebra that the minor and major axes of $\mathcal S_k$ have orthogonal preimages in $\mathcal S_0$ (see Remark \ref{rmk:SVD} below), and these preimages are precisely $e^{(k)}$ and $f^{(k)}$. For any $i \geq 1$, we let 
\begin{equation}\label{eq:eik}
e^{(k)}_{i}:=  D\Phi^{i}(e^{(k)}) 
\quand 
 f^{(k)}_{i} :=D\Phi^{i}(f^{(k)}).
  \end{equation}
 Notice that  \( e^{(k)}_{i},  f^{(k)}_{i} \in T_{\xi_{i}}M \),  where  \(\xi_{i}=\Phi^{i}(\xi_{0})\), and  
\(
e^{(k)}_{k}
\) and 
\(  f^{(k)}_{k}
  \)
  are by definition minor and major semi-axes of the ellipse \( \mathcal S_k\) and so are also \emph{orthogonal} (which is \emph{not generally the case} when \( i \neq k\)).
 Normalizing these vectors we can define an orthonormal basis in  \( T_{\xi_k}M\) given by the unit vectors 
  \[
 \mathcal H^{(k)}_{k} :=  \{{e^{(k)}_{k}}/{\|e^{(k)}_{k}\|}, {f^{(k)}_{k}}/{\|f^{(k)}_{k}\|} \}. 
\]
In coordinates  \( \mathcal H^{(k)}\) in \( T_{\xi_{0}}M \) and \( \mathcal H^{(k)}_k\) in \( T_{\xi_{k}}M \), the derivative 
\(
 D\Phi^{k}_{\xi_{0}}: T_{\xi_{0}}M \to T_{\xi_{k}}M
 \)
has   \emph{diagonal} form 
\begin{equation}\label{eq:hypcorest}
D\Phi^{k}_{\xi_{0}} = \begin{pmatrix} \|f^{(k)}_{k}\| & 0 \\ 0 &  \|e^{(k)}_{k}\|\end{pmatrix} =
\begin{pmatrix}   \|D\Phi^{k}_{\xi_{0}}\|  & 0 \\ 0 & \|(D\Phi^{k}_{\xi_{0}})^{-1}\|^{-1}\end{pmatrix}. 
\end{equation}
This diagonal form of the derivative can be very useful in a number of situations. 

\begin{remark}\label{rmk:SVD}
In view of \eqref{eq:hypcorest}, hyperbolic coordinates in smooth dynamics correspond to the \emph{singular value decomposition} of the linear operators $D\Phi_{\xi_0}^i : T_{\xi_0}M \to T_{\Phi^i(\xi_0)}M$. In general, a linear map $A : \mathbb R^n \to \mathbb R^m$ is expressible as $A = U \Sigma V^*$, where $U \in \mathrm O(m), V \in \mathrm O(n)$ are orthogonal matrices and $\Sigma$ is an $m \times n$ matrix whose non-diagonal entries are all $0$. The eigenvalues of $A^* A$ are the squares of the diagonal entries of $\Sigma$ (the ``singular values''), where $A^*$ is the adjoint of $A$, and the corresponding eigenvectors are the columns of $V$. In $\mathbb R^2$, these eigenvectors are the directions of maximal and minimal expansion. So, $e^{(k)}$ and $f^{(k)}$ are eigenvectors of the matrix $(D\Phi^k_{\xi_0})^* \circ D\Phi_{\xi_0}^k$, where $(D\Phi_{\xi_0}^{k})^* : T_{\Phi^k(\xi_0)}M \to T_{\xi_0}M$ is the adjoint of $D\Phi_{\xi_0}^k$ with respect to the Riemannian inner product in $T_{\xi_0}M$ and $T_{\Phi^k(\xi_0)} M$, and these eigenvectors have corresponding eigenvalues $\|e^{(k)}_k\|^2$ and $\|f^{(k)}_k\|^2$.
\end{remark}

\subsubsection{Finite-time stable and unstable manifolds}
\label{sec:finitetime}
A second important observation is that we can extend hyperbolic coordinates to a neighbourhood of the base point \( \xi_{0}\) since, if \( C_{\xi_{0}, k}  < 1 \) at \( \xi_{0}\) then, since \( \Phi^{k}\) is assumed to be \( C^{1}\), the same will be true in a neighbourhood of \( \xi_{0}\).  There exist therefore in this neighbourhood two \emph{orthogonal unit vector fields} \( e^{(k)}, f^{(k)}\) given by the most contracted and most expanded direction at each point. Moreover, hyperbolic coordinates  can be computed explicitly in terms of the partial derivatives of \( D\Phi^{k}\): parametrizing the unit circle by \( \mathcal S = \{(\sin \theta, \cos \theta), \theta \in [0, 2\pi)\}\), the angles  which map to the minor and major axes of the ellipse  \( \mathcal S_k\) are solutions to the  equation
\(
{d}\|D\Phi^{k}  (\sin\theta, \cos\theta)\|/{d\theta}=0,
\)
which gives
\begin{equation}\label{c4}
\tan 2\theta  =
\frac{2 (\pfi x1^{k}\pfi y1^{k} +\pfi x2^{k}\pfi y2^{k})}
{(\pfi x1^{k})^2+(\pfi x2^{k})^2 - (\pfi y1^{k})^2 -(\pfi y2^{k})^2}.
\end{equation}
This gives an alternative proof of the fact that \( e^{(k)}, f^{(k)}\) are  orthogonal  and also shows that they depend on the base point with \emph{the same regularity as the partial derivatives of} \( \Phi^{k}\). 
In  particular, if  \( \Phi^{k} \) is \( C^{2} \) in a neighbourhood of \( \xi_{0}\) then the unit vectors  \( e^{(k)}, f^{(k)}\) define two orthogonal \( C^{1}\) vector fields and are therefore locally \emph{integrable} and define two orthogonal \emph{foliations} \cal Ek, \cal Fk. The leaves of these foliations are  the \emph{integral curves of the most contracted and most expanded directions} for \( D \Phi^{k}\) and therefore can naturally be thought of as  \emph{(finite time) stable  and unstable manifolds}  (of order \(k \)). This idea has been developed in \cite{HolLuz05, HolLuz06} to give new proofs of the classical stable manifold theorems in certain two-dimensional settings, including for orbits which exhibit very weak forms of hyperbolicity.

Extending the notation introduced in \eqref{eq:eik} above, we let 
 \(
  \cal Ek_{i}:=\Phi^{i}(\cal Ek)
\) 
and \( 
   \cal Fk_{i}:=\Phi^{i}(\cal Fk) 
 \)
 denote the images of these stable and unstable foliations, which are themselves the foliations given by the integral curves of the vector fields \( e^{(k)}_{i}, f^{(k)}_{i} \). In particular,  the foliations \(  \cal Ek_{k}, \cal Fk_{k}  \) are  orthogonal and therefore  we can use the diagonal form of the derivative given in \eqref{eq:hypcorest} in a neighbourhood of the point \( \xi_{0}\).

\subsection{Quasi-hyperbolicity}\label{sec:quasihyp}

As we have seen in the previous section, hyperbolic coordinates give rise to some  dynamically significant geometric structures, in particular the orthogonal foliations in which the derivative has the especially simple diagonal form \eqref{eq:hypcorest}. However, the usefulness of the coordinates depends on how much information we have about these foliations, such as the direction of the leaves and their curvatures. In principle, a lot of information can be obtained from the formula in  \eqref{c4}, but in practice this can really be used only for the first iterate \( k =1 \), as we do not generally have enough explicit information about the partial derivatives for higher iterates. We therefore need to take a different approach which uses somewhat ``coarser'' information about the derivative along the orbit, but is still sufficient, in some cases, to deduce relevant bounds for the geometry of the hyperbolic coordinates and the corresponding foliations.

 We  formulate a notion of \( \mathfrak C \)\emph{-quasi-hyperbolicity} along the orbit of a point in terms of  a set \( \mathfrak C \) of constants. The conditions involved in this formulation may appear at first sight somewhat  technical, but are  in fact quite natural and quite mild. 
 While our goal is to formulate this notion for \emph{singular} systems with unbounded derivative, our results are also highly relevant in the simpler setting of non-singular systems in which the derivative is uniformly bounded. In the non-singular situation, the formulation is a bit simpler, so for the sake of clarity, we formulate our definitions in the non-singular setting first.  
  
  \subsubsection{Quasi-hyperbolicity in non-singular systems}
\begin{definition}\label{def:nonsingular-QHv2}
Given a set  \( \mathfrak C = 
\{ \Gamma,   \lambda, b, c \}\) of positive constants, 
the point $\xi_0$ is \(\mathfrak C\)-\emph{quasi-hyperbolic} at time $k$ if  there exists constants \( C > 0, B, D \geq 1 \) such that for every \( 1 \leq i \leq k \) the map \(\Phi^{i}\) is \(C^{2}\) at \( \xi_{0}\) and  satisfies 
\begin{equation}\label{eq:constants-basic-nonsing}
 C \lambda^{i} < \|D\Phi^{i}_{\xi_{0}}\| <  D\Gamma^{i}
 \quand
C_{\xi_0,i} := \frac{\|(D\Phi_{\xi_{0}}^{i})^{-1}\|^{-1}}{\|D\Phi_{\xi_{0}}^{i}\|} < B c^{i}  <1,
\end{equation}
and 
\begin{equation}\label{eq:constants-basic2-nonsing} 
    \|D\Phi_{\xi_{i-1}}\|, \|D^2\Phi_{\xi_{i-1}}\| <  D \Gamma
 \quand
     \det D\Phi_{\xi_{i-1}}  \leq b
\end{equation}
with the constants satisfying 
 \begin{equation}\label{eq:nonsing}
\Gamma \geq  \max\{\lambda, 1\},   \quad\quad  b < \lambda^{2}, \quad\quad  c< \lambda^{2}/\Gamma^{2} < 1. 
  \end{equation}
\end{definition}
We make several remarks  about the interpretation and significance of these conditions before stating their generalization to the singular setting. 

\begin{remark}\label{rem:2.2}
The bounds in \eqref{eq:constants-basic-nonsing} are the two core ``\emph{hyperbolicity}'' conditions, albeit, and crucially,  formulated in a way that does not require an \emph{a priori} decomposition of the tangent bundle.  
\end{remark}

\begin{remark}\label{rem:ratio}
 The lower and upper bounds on \( \|D\Phi^{i}_{\xi_{0}}\|\) in \eqref{eq:constants-basic-nonsing} are in some sense ``trivial'' since such bounds always exist, but the purpose here is to give these bounds in terms of specific constants which appear also in the other conditions (notice that they imply in particular a minimum ``growth'' of the norm of the derivative but we only assume \( \lambda > 0 \), not necessarily \(\lambda > 1 \),  so this may not necessarily require actual  growth).  The ratio \( \lambda/\Gamma\), which is always \( \leq 1 \), may in some situations be  chosen very close to 1.  For instance, suppose that \( \xi_{0}\) is a typical point for an invariant probability measure \( \mu \), ``typical'' in the sense that the Lyapunov exponent \( \chi  = \lim_{n\to\infty}n^{-1} \log \|D\Phi^{i}_{\xi_{0}}\|\) is well defined. 
 This implies that for any \( \epsilon > 0 \) and suitable constants \( C\) and \( D \) (depending on \( \epsilon\)), the bounds on \( \|D\Phi^{i}_{\xi_{0}}\|\) in \eqref{eq:constants-basic-nonsing} are satisfied with \( \lambda = e^{\chi -\epsilon}\) and \( \Gamma = e^{\chi + \epsilon}\). Then \( \lambda/\Gamma = e^{-2\epsilon} \), which can be made \emph{arbitrarily close to} 1 by taking \( \epsilon\) small. 
\end{remark}

\begin{remark}
We also emphasize that we do not assume \( \lambda > 1 \), and 
the second set of inequalities in \eqref{eq:constants-basic-nonsing} essentially say there is is also a ``contracting'' direction, albeit just contracting relative to some ``more expanding'' direction (which may not even be expanding). This is thus essentially a weak ``dominated decomposition'' condition. Notice that the bound is formulated in terms of the constant \( c \) which is bounded above by the ratio \( \lambda^{2}/\Gamma^{2}\). This  puts some restrictions on its range of applicability but, as mentioned in Remark \ref{rem:ratio}, there are many cases in which \( \lambda \) and \( \Gamma \) can be chosen so that \( \lambda/\Gamma\) is very close to 1, allowing this condition to be quite easily satisfied. 
\end{remark}

\begin{remark}
The conditions in \eqref{eq:constants-basic-nonsing} and \eqref{eq:constants-basic2-nonsing}  could morally be stated directly in terms of \( \lambda \) and \( \Gamma\), without reference to the constants $b$ and $c$,  but for technical reasons we require some uniform bounds independent of \( k \) which are achieved by introducing the constants \( b, c \), which can be thought of as ``arbitrarily close''  to \( \lambda^{2}\) and \( \lambda^{2}/\Gamma^{2}\) respectively. 
\end{remark}

\begin{remark}
The non-singularity of the map is reflected in the uniform upper bound for the norms of the first and second derivatives in the first expression in \eqref{eq:constants-basic2-nonsing}. We will have to relax this in the general setting. 
 \end{remark}
 
 \begin{remark}
The determinant is \emph{not} required to be small. In many cases we have \( \lambda > 1 \) and therefore the bound \eqref{eq:constants-basic2-nonsing} on \( b \)  is not very restrictive at all, allowing us to apply our results even to area-preserving systems. 
 \end{remark}

\begin{remark}
Strictly speaking the set of constants \( \mathfrak C \) which define quasi-hyperbolicity also includes the constants \(  B, C , D\). These latter constants will come into the definition of some constants which appear in our results, but there are no restrictions on them for the definition of quasi-hyperbolicity. Therefore, for clarity, we have not included them in the ``core'' constants \( \mathfrak C \). 
\end{remark}

\begin{remark}
    The assumption that $B, D \geq 1$ (as opposed to simply $B, D > 0$) is an assumption based on technical convenience. Since $B$ and $D$ are used in upper bounds, we lose no generality in assuming $B, D \geq 1$. This also holds for $B$ and $D$ in Definition \ref{def:nonsingular-QHv2-I} below.
\end{remark}

\subsubsection{Quasi-hyerbolicity in  singular systems}
We now generalize the definition above to singular systems in which the derivative may be unbounded.

\begin{definition}\label{def:nonsingular-QHv2-I}[(Singular) Quasi-Hyperbolicity]
Given a set \( \mathfrak C = \{\Gamma,  \widetilde\Gamma, \lambda, b, c, \tilde c \}\) of positive constants, 
the point $\xi_0$ is \(\mathfrak C\)-\emph{quasi-hyperbolic} at time $k$ if  there exists constants \(B, D\geq 1 \geq \widetilde B, C >0\)
such that for every \( 1 \leq i \leq k \) the map \(\Phi^{i}\) is \(C^{2}\) at \( \xi_{0}\) and  
\begin{equation}\label{eq:constants-basic-sing1}
 C \lambda^{i} < \|D\Phi^{i}_{\xi_{0}}\| <  D\Gamma^{i}
 \quand 
C_{\xi_0,i} := \frac{\|(D\Phi_{\xi_{0}}^{i})^{-1}\|^{-1}}{\|D\Phi^{i}\|} \leq B c^{i}  <1,
\end{equation}
and 
\begin{equation}\label{eq:constants-basic2-sing1}
    \|D\Phi_{\xi_{i-1}}\|, \|D^2\Phi_{\xi_{i-1}}\| <  D \Gamma \widetilde\Gamma^{i-1},
    \quand
     \det D\Phi_{\xi_{i-1}}  \leq b.
\end{equation}
We assume moreover that 
    \begin{equation}\label{eq:constants-basic-x}
 \widetilde \Gamma \geq 1 \quand  \Gamma >  \max\{\lambda, 1\}  \quand b < \Gamma^2 \widetilde\Gamma
    \end{equation}
    and \emph{either}
        \begin{equation}\label{eq:condition-I}\tag{I}
 b < \lambda^{2}/\widetilde \Gamma \quand c < \lambda^{3}/\Gamma^{3}\widetilde\Gamma^{3}<1,
    \end{equation}
    in which case we say that  $\xi_0$ is  \(\mathfrak C\)-\emph{quasi-hyperbolic}  of type (I), 
    \emph{and/or}
   \begin{equation}\label{eq:condition-II}\tag{II}
 C_{\xi_{i-1},1} := \frac{\|(D\Phi_{\xi_{i-1}})^{-1}\|^{-1}}{\|D\Phi_{\xi_{i-1}}\|} \geq \widetilde B \tilde c^{i-1}
  \quand   b < \lambda^{2}\tilde c \quand c < \lambda^{2}\tilde c^{2} /\Gamma^{2}\widetilde\Gamma < \tilde c  \leq 1 
    \end{equation}
        in which case we say that  $\xi_0$ is  \(\mathfrak C\)-\emph{quasi-hyperbolic}  of type (II). 
\end{definition}

We conclude this section with a number of additional remarks concerning our assumptions in the singular setting. These remarks are not formally needed for the statement of our results in Section \ref{sec:results}, but are included to help clarify how the assumptions should be interpreted heuristically.
  
 \begin{remark}
 The two core sets of conditions \eqref{eq:constants-basic-sing1} and \eqref{eq:constants-basic2-sing1} are exactly identical to the conditions \eqref{eq:constants-basic-nonsing} and  \eqref{eq:constants-basic2-nonsing} respectively in the non-singular case \emph{except} for the addition of the new constant \(\widetilde\Gamma\)  in \eqref{eq:constants-basic2-sing1}, which  now  allows the derivative to be unbounded along the orbit, albeit in a controlled way. This is a significant generalization of the definition and hugely increases the range of systems to which it is applicable. 
 The conditions on the constants in \eqref{eq:constants-basic-x} and \eqref{eq:condition-I} are also very similar  to the corresponding conditions \eqref{eq:nonsing} in the non-singular setting albeit incorporating the new constant \( \widetilde \Gamma \). The alternative condition \eqref{eq:condition-II} is not just a formal condition on the constants but introduces a requirement of a specific lower bound on the \emph{pointwise co-eccentricity} $C_{\xi_i,1}$. 
\end{remark}

\begin{remark}
The distinction between type (I) and type (II) singular hyperbolicity is not particularly relevant from a conceptual point of view. It is rather just a technical distinction motivated by the fact that we can address both situations by estimating some expressions in slightly different ways in the course of the proof, and one or the other might be easier to verify in some specific examples.  The results we obtain are the same: they do not distinguish between these two cases except in the specific values of some of the constants. 
 \end{remark}

 \begin{remark}
 The bound on the one-step co-eccentricity in \eqref{eq:condition-II} is  essentially just a mild  \emph{bounded recurrence} condition for the orbit near the singularity, as are also the pointwise bounds in \eqref{eq:constants-basic2-sing1}. 
The constants \( \tilde c\) and \( \widetilde\Gamma\) are therefore related and we can even  choose them satisfying an explicit relationship, such as   \( \tilde c = 1/\widetilde \Gamma^{1/2}\). Using the fact that \( \Gamma > \lambda \),  this would imply 
  \(
c  <   {\lambda^{2} \tilde c^{2}}/{\Gamma^{2} \widetilde \Gamma} 
<1/{\tilde\Gamma^{2}}   = \tilde c^{4}  \leq  \tilde c \leq 1
\) which shows that this choice is compatible with the last set of inequalities in \eqref{eq:condition-II}.
 \end{remark}
 
 \begin{remark}
Taking \( \tilde c = \widetilde\Gamma = 1\) in \eqref{eq:condition-II} we recover \emph{exactly} the conditions \eqref{eq:nonsing} of the non-singular setting of Definition~\ref{def:nonsingular-QHv2} (the condition on the pointwise co-eccentricity and the constant \( \widetilde B \) do not appear explicitly there but are automatically satisfied). We will therefore not give a separate proof of our results in the non-singular case since they are included as special cases of the singular case of type (II).  Also, to simplify the terminology we will usually omit explicit reference to the set  \( \mathfrak C \) since this is  understood to have been fixed. 
 \end{remark} 

\begin{remark}
    The constant $c$ is an \emph{upper} bound of the the \emph{accumulated} co-eccentricity of $D\Phi$ along the orbit $\xi_i$, whereas $\tilde c$ is a \emph{lower} bound of the \emph{one-step} co-eccentricity. The assumption that $\tilde c > c$ is not contradictory with the fact that $c$ is an upper bound, while $\tilde c$ is a lower bound, for two reasons. Firstly, due to rotation effects, if $A_1$ and $A_2$ are two matrices with co-eccentricities $C_{A_1}$ and $C_{A_2}$, then there is no relation between the product of the co-eccentricities $C_{A_1} C_{A_2}$ and the co-eccentricity of the product $C_{A_1 A_2}$. So there need be no relation between the accumulated co-eccentricity $C_{\xi_0,i}$ and the one-step co-eccentricity $C_{\xi_i,1}$. Secondly, if there are no rotation effects, and the accumulated co-eccentricity is the product of the one-step co-eccentricities (as in the classical geometric Lorenz attractor), there is still no contradiction: since $C_{\xi_0,i} \leq Bc^i$ and $C_{\xi_i,1} \geq \widetilde B\tilde c^i$, if $C_{\xi_0,i} = \prod_{j=0}^{i-1} C_{\xi_j,1}$, we would have: 
    \[
    Bc^i \geq C_{\xi_0,i} = \prod_{j=0}^{i-1} C_{\xi_j,1} \geq \prod_{j=0}^{i-1} \widetilde B\tilde c^j = \widetilde B^i \tilde c^{i(i-1)/2}.
    \]
For appropriate choices of $B$ and $\widetilde B$, it is therefore perfectly reasonable to suppose $\tilde c > c$. 
\end{remark}

\subsubsection{Auxiliary Constants}
The statements of our main results below,  as well as the intermediate computations in the argument, will involve a number of lengthy expressions involving the constants used in the definition of quasi-hyperbolicity. To simplify these expressions we will introduce a number of auxiliary constants at various steps of the proof. For ease of reference we  collect the definitions of all these constants here. 
First of all, let
\begin{equation}\label{eq:sing-constants-v0}
Q_0 := \sqrt{\frac{2}{1-B^2 c^2}}
\quand K_{1} := \frac{Q_0^2}{\sqrt 2} 
%\frac{\sqrt 2}{1-B^{2} c^{2}}.
\end{equation}
Notice that by  \eqref{eq:constants-basic-sing1}, we have \(  Bc < 1 \) and  therefore also  \(  B^{2}c^{2} < 1 \)
and so  \( Q_{0}\) and \( K_{1}\)  are well-defined positive constants.  Moreover, assuming the constant \( B \) is fixed, we have that  \( Q_{0}, K_{1}\to \sqrt 2\) as \(c\to 0 \). 
Then we let 
\begin{equation*}\label{eq:sing-constants-v1}
%\todelete{Q_0 = \sqrt{\frac{2}{1-B^2 c^2}},} \quad 
Q_1 :=  BD + \frac{Q_0 BD^3\Gamma}{C(\lambda - \Gamma \widetilde \Gamma c)}, 
\quad  Q_2 := \frac 1 C + \frac{Q_0 D^2\Gamma\lambda }{C^2(\lambda^2 - \widetilde \Gamma b)},
\quad 
Q_{3}:=  \frac{Q_1D\Gamma^2 \widetilde \Gamma}{\lambda},
\quad 
Q_{4}:=\frac{Q_1 Q_2 D\Gamma^5 \widetilde \Gamma^4 }{\lambda^2(\lambda^3 - \Gamma^3 \widetilde \Gamma^3 c)}.
\end{equation*}
%\todelete{
%\begin{equation}\label{eq:Q3Q4}
%Q_{3}:=  \frac{Q_1D\Gamma^2 \widetilde \Gamma}{\lambda},
%\qquad 
%Q_{4}:=\frac{Q_1 Q_2 D\Gamma^5 \widetilde \Gamma^4 }{\lambda^2(\lambda^3 - \Gamma^3 \widetilde \Gamma^3 c)}.
%\end{equation}
%}

These will be used in the setting of type \eqref{eq:condition-I} quasi-hyperbolicity. 
Then, by \eqref{eq:constants-basic-x} we have $\lambda/\Gamma \widetilde \Gamma<1$ and therefore~\eqref{eq:condition-I}  gives $c < \lambda^3/\Gamma^3\widetilde\Gamma^3 < \lambda/\Gamma \widetilde \Gamma<1$ which implies that \( \lambda > \Gamma \widetilde \Gamma c\) and therefore  the denominators in the definition of \( Q_{1}\)  and \( Q_{4}\) are strictly positive. Similarly,  from  \eqref{eq:condition-I} we have $\lambda^2 > \widetilde \Gamma b$ which implies that the denominator in the definition of \( Q_{2}\)  is strictly positive. 
It follows that under the assumptions  \eqref{eq:condition-I},  $Q_1$-$Q_4$ are all well-defined positive constants. 
Moreover, since \( Q_{0}\to \sqrt 2\) as \( c\to 0 \) it follows that  $Q_1,  Q_2, Q_{4}$ are  monotonic in $c$ and decrease to positive constants as $c \to 0$, whereas \(Q_{3}\) is independent of \( c \).

We now let 
\begin{equation*}\label{eq:sing-constants-v2}
    \widetilde Q_1 := BD + \frac{Q_0B\tilde c}{\widetilde B(\tilde c - c)}, 
    \quad \widetilde Q_2 := \frac 1 C + \frac{Q_0D \lambda^2 \tilde c}{\widetilde B C^2 (\lambda^2 \tilde c - b)}.
    \quad 
    \widetilde Q_{3}:=\widetilde Q_1D\Gamma,
\quad 
\widetilde Q_{4} := \frac{\widetilde Q_1 \widetilde Q_2 D \Gamma^4 \widetilde\Gamma }{\lambda^2 (\lambda^2 \tilde c^2 - \Gamma^2 \widetilde \Gamma c)}.
\end{equation*}
These will be used in the setting of type \eqref{eq:condition-II} quasi-hyperbolicity, in which case it follows 
from    \eqref{eq:condition-II} 
that $\widetilde Q_1$-$\widetilde Q_4$ are  are well-defined and positive. Moreover, \( \widetilde Q_1, \widetilde Q_{2}, \widetilde Q_4$  are  monotonic in $c$ and decrease to positive constants as $c \to 0$, whereas \( \widetilde Q_{3}\) is independent of \( c \).
Finally we let 
\begin{equation}\label{eq:Q}
Q:= \frac{BD^3\Gamma^4 \widetilde\Gamma}{C^2 \lambda^2(\Gamma^2\widetilde\Gamma - b)}
\quand 
 K_{2} := \max\{K_{1}(Q_{3}+Q_{4}+Q), K_{1}(\widetilde Q_{3}+ \widetilde Q_{4}+Q)\}
\end{equation}
By condition \eqref{eq:constants-basic-x}, \( Q \) is a well defined positive constant and is clearly independent of \( c \). 
%\comment{SL: We don't actually seem to have \( b < \Gamma^2\widetilde\Gamma \) anywhere. Should we add it as a condition? (DV: Well spotted. Added. But it is enough to assume $\Gamma^2\widetilde\Gamma \neq b$, and replace this term in $Q$ with $|\Gamma^2\widetilde\Gamma - b|$. Possibly a better edit? This comes from \eqref{eq:fbound-pf-5I}).} 
Therefore \( K_{2}\) is also  positive and decreases to a positive constant as \( c \to 0\).

\subsection{Statement of Results}\label{sec:results}
We now give our two main results on the properties of hyperbolic coordinates for quasi-hyperbolic orbits. 

\subsubsection{Convergence of hyperbolic coordinates}
Our first result concerns the dependence of hyperbolic coordinates on the iterate \( k \). Notice that a-priori \emph{there need not be any relation at all between the hyperbolic coordinates at time \( k \) and at time  \( k+1\)}. Indeed,  recall from Section \ref{sec:hypcoord} that \( e^{(k)},  f^{(k)} \) are the pre-images of semi-axes of the ellipse \( \mathcal S_{k} = D\Phi^{k}(\mathcal S) \) and  \( e^{(k+1)},  f^{(k+1)} \) are the pre-images semi-axes of the ellipse \( \mathcal S_{k+1}=D\Phi^{k+1}(\mathcal S) \). 
Since  \( \mathcal S_{k+1}= D\Phi_{\xi_{k}}(\mathcal S_{k})\), it  is easy to construct examples in which the major and minor axes of \( \mathcal S_{k}\) are mapped by 
\( D\Phi_{\xi_{k}} \) to pretty much any desired position in \( \mathcal S_{k+1}\). An extreme case would be for \( D\Phi_{\xi_{k}} \) to map the major (resp. minor) axis of \( \mathcal S_{k}\)  to the minor (resp. major)  axis  of \( \mathcal S_{k+1}\), implying that the most contracting (resp. most expanding) vector under \( D\Phi^{k} \) is the most expanded (resp; most contracted) vector by \( D\Phi^{k+1} \), in which case we have \( e^{(k+1)}=f^{(k)}\) and \( f^{(k+1)}=e^{(k)}\). 

This shows that in principle  hyperbolic coordinates can change  wildly for different values of \( k \), which can make it very difficult to use them in any effective way. However, there are (at least) two ways to control such ``erratic'' changes. The first is by assuming the existence of some ``hyperbolic conefield'' that guarantees that at every step the derivative maps ``expanding directions'' to ``expanding directions'', thus avoiding the possibility of ``switching'' the most contracted and most expanded vectors as described above. The existence of such conefields, however, is a quite strong assumption, which is not generally satisfied. A more general approach, and the focus of our results, is based on the observation that if the co-eccentricity of \( D\Phi^{k} \) is very small, then the ellipse \( \mathcal S_{k}\)  is very ``thin'' (in the sense that the ratio between the minor and major axes is very small), and  \( D\Phi_{\xi_{k}} \) would have to have even smaller co-eccentricity  to  switch the contracting and expanding directions since it would have to map the minor axis of \( \mathcal S_{k}\) to a vector whose norm is larger than the image of the major axis. Some of the conditions on the definitions of quasi-hyperbolicity are precisely motivated by the use of this approach in order to control the fluctuation of the hyperbolic coordinates. We will prove the following. 

\begin{theorem}\label{thm:sing-hyp-coord-cvgce}
There are constants \( Q_{1}, \widetilde Q_{1}\) 
%depending only on  \( \mathfrak C = \{ B, \widetilde B, C, D, \Gamma,  \widetilde\Gamma, \lambda, b, c, \tilde c \}\) 
such that for every \( k \geq 1 \) and every \( 1\leq i \leq k \), if  $\xi_0$ is  quasi-hyperbolic up to time $k$ of type \eqref{eq:condition-I}, then 
    \begin{equation}\label{eq:cvgce-k-quasi-hyp-1-thm}
\|e^{(k)} - e^{(i)}\| \leq Q_1 \left( \frac{\Gamma\widetilde\Gamma c}{\lambda}\right)^i,
\end{equation}
while if $\xi_0$ is quasi-hyperbolic up to time $k$ of type  \eqref{eq:condition-II}, then 
\begin{equation}\label{eq:cvgce-k-quasi-hyp-sac-1-thm}
\|e^{(k)} - e^{(i)}\| \leq \widetilde Q_1\left( \frac{c}{\tilde c}\right)^i. 
\end{equation}
In particular, in the non-singular setting, where we can take \( \tilde c =1\),  we have 
\begin{equation}\label{eq:convnonsing}
 \|e^{(k)} - e^{(i)}\| \leq \widetilde Q_1  c^{i}
\end{equation}
\end{theorem}

\begin{remark}
Condition \eqref{eq:condition-I} says that \( c < (\lambda/\Gamma\widetilde\Gamma)^{3} < 1 \) which implies  \( c < \lambda/\Gamma\widetilde\Gamma  \) and therefore \( \Gamma\widetilde \Gamma c/\lambda <1\), and Condition \eqref{eq:condition-II} says that \( c< \tilde c\), and therefore all 3 in \eqref{eq:cvgce-k-quasi-hyp-1-thm}, \eqref{eq:cvgce-k-quasi-hyp-sac-1-thm}, \eqref{eq:convnonsing}, are  decreasing exponentially in \( i \). 
\end{remark}

\begin{remark}
The expression  \eqref{eq:convnonsing} captures, in its simplest form, the ``spirit'' of this result and, to some extent, the main motivation for the definition of quasi-hyperbolicity. Since  \( c\in (0,1)\), this implies that the sequence of hyperbolic coordinates form a Cauchy sequence and therefore \emph{converge} as \( k \to \infty \) as long as \( \xi_{0}\) is quasi-hyperbolic for all \(k \geq 1 \). Conditions  \eqref{eq:cvgce-k-quasi-hyp-1-thm}-\eqref{eq:cvgce-k-quasi-hyp-sac-1-thm} imply  the same in the singular case.
\end{remark}

\begin{remark}
A bound similar to   \eqref{eq:convnonsing} was proved in \cite{BenCar91, MorVia93, Via97ihes} in terms of the bound \( b \) for the \emph{determinant}. In these papers  the determinant is always assumed to be small and therefore this bound would not apply to certain systems, for example to area-preserving maps. We have here that the determinant is not in fact the natural quantity to bound this convergence but rather the co-eccentricity, which can be \( < 1 \), and possibly very small, even for area-preserving maps.  
\end{remark}

\begin{remark}
The bounds in \eqref{eq:cvgce-k-quasi-hyp-1-thm} and \eqref{eq:cvgce-k-quasi-hyp-sac-1-thm} are formulated in terms of \( i \). This means that no matter how large \( k \geq i \) is, as long as \( \xi_{0}\) is quasi-hyperbolic up to time \( k \) for the same given set of constants \( \mathfrak C \), the hyperbolic coordinates of order \( k \) must remain within a fixed ``cone'' around the hyperbolic coordinates of order \( i \).  In particular, if we can compute or estimate the direction \( e^{(1)}\) using the explicit formula \eqref{c4} then \eqref{eq:cvgce-k-quasi-hyp-1-thm} and \eqref{eq:cvgce-k-quasi-hyp-sac-1-thm} give bounds on the possible positions of all ``future'' contracting directions \( e^{(k)}\). 
\end{remark}

\begin{remark}
The results above are stated for the most contracting directions \( e^{(i)}, e^{(k)}\) but since hyperbolic coordinates are always orthogonal, exactly the same statements clearly hold for  \( f^{(i)}, f^{(k)}\). 
\end{remark}

\subsubsection{Derivative of hyperbolic coordinates}
To introduce our second main result, recall the expression in  \eqref{c4}, which shows that the hyperbolic coordinates depend \( C^{1}\) on the base point \( \xi_{0}\). We can therefore consider the derivatives \( D e^{(k)} \) and \( D f^{(k)} \) of the hyperbolic coordinates with respect to the base point  (notice that \( De^{(k)} =  D f^{(k)} \) since  \( e^{(k)} \) and \( f^{(k)}\) are always orthogonal). This derivative, and in particular the norm of this derivative, is of interest as it has several implications, for example for the geometry of the local foliations given by the integral curves of the unit vector fields defined by  \( e^{(k)} \) and \( f^{(k)}\) (recall the discussion in Section \ref{sec:finitetime}).  We show that this norm is uniformly bounded in \( k \) by a constant that essentially depends on on the constant \( c \) which bounds the co-eccentricity.

\begin{theorem}\label{thm:main-phase} 
There are constants $K_1$ and $K_2$ 
%which depend only on \( \mathfrak C = \{ B, \widetilde B, C, D, \Gamma,  \widetilde\Gamma, \lambda, b, c, \tilde c \}\) 
 such that for every  $k \geq 1$, if  $\xi_0$  is a  quasi-hyperbolic point up to time $k$, then for \( \varsigma = x,y \), we have 
\begin{equation}\label{eq:main-phase}
\|D_{\xi_0} e^{(k)}\| \leq K_1 \|D^2 \Phi_{\xi_0}(e^{(1)},\cdot)\| + K_2 c
\leq K_{1} \sqrt 2  \| \partial_\varsigma D\Phi_{\xi_0}  e^{(1)}\| + K_2c.
\end{equation}
Furthermore, $K_1$ and $K_2$ monotonically decrease to nonzero constants as $c \to 0$. %Specifically, $K_1, K_2$ are $O\left((1-c^2)^{-1/2}\right)$ as $c \to 0$.
\end{theorem}

\begin{remark}
We emphasize that the constants $K_{1}$ and \( K_{2}\), defined explicitly in \eqref{eq:sing-constants-v0} and \eqref{eq:Q} above, are independent of \( k \) and just depend on the constants in $\mathfrak C$ and on $B, \tilde B, C, D$ in the definition of quasi-hyperbolicity. In particular, the variation of the hyperbolic coordinates of arbitrarily high order is uniformly bounded. 
\end{remark}

\begin{remark}
The first term $\|D^2\Phi_{\xi_0}(e^{(1)},\cdot)\|$ in  \eqref{eq:main-phase}  involves the second derivative of the map \( \Phi \). It is the operator norm for the linear map $v \mapsto D^2 \Phi_{\xi_0}\left(e^{(1)}, v\right)$. In other words, this describes the variation of the action of $D\Phi_{\xi}$ on the vector field $e^{(1)}(\xi)$. It  depends only on the first iterate of \( \Phi\) and  is \emph{coordinate-free} as its formulation does not presuppose any a-priori choice of coordinate systems. 
In practice, however, estimating $\|D^2\Phi_{\xi_0}(e^{(1)},\cdot)\|$ may require working in some specific choice of coordinates in which case the 
second bound in \eqref{eq:main-phase} is more useful. Indeed,  we can then use information about the first order partial derivatives of \( \Phi\) to estimate the position of \( e^{(1)}\) using \eqref{c4} and then information about the second order partial derivatives to estimate \( \| \partial_\varsigma D\Phi_{\xi_0}  e^{(1)}\| \).  For example,  if we choose a coordinate system where $D\Phi$ is ``mostly contracting'' in the vertical direction, so that $f^{(1)} \approx (1,0)$ and $e^{(1)} \approx (0,1)$, then $D^2\Phi_{\xi_0}(e^{(1)}, \cdot)$ is a linear map approximated by the matrix 
\begin{equation}\label{eq:second-deriv-approx}
D^2\Phi_{\xi_0}(e^{(1)}, \cdot) \approx \left( \begin{array}{cc}
\partial_{xy} \Phi_1(\xi_0) & \partial_{yy} \Phi_1(\xi_0) \\
\partial_{xy} \Phi_2(\xi_0) & \partial_{yy} \Phi_2(\xi_0)
\end{array}\right) = \partial_y (D\Phi_{\xi_0})
\end{equation}
If the map $\Phi$ is a $C^2$ perturbation of a one dimensional map, as in the strongly dissipative H\'enon maps of \cite{BenCar91},  then $\partial_{\varsigma y}\Phi_j(\xi_0)$ is small for $\varsigma = x,y$ and $j=1,2$ and  $\|D^2\Phi_{\xi_0}(e^{(1)},\cdot)\|$ is bounded by a small constant. 

\end{remark}

\begin{remark}
The second term \( K_{2}c\) in  \eqref{eq:main-phase} is arguably the most important part of the statement as \emph{it highlights the significance of the co-eccentricity constant} \( c \).  Previous estimates of the variation of hyperbolic coordinates have always been formulated in terms of the bound \( b \) for the \emph{determinant}, and moreover have assumed that this bound was ``\emph{sufficiently small}''. A main innovation in our results is to observe that the co-eccentricity is the key quantity in these estimates, not the determinant. In particular this allows us to apply the results to systems in which the determinant is not necessarily small, even area-preserving systems. 
\end{remark}

\subsection{Overview of the Proof}

In Section \ref{sec:hypcoords-cvgce} we prove Theorem \ref{thm:sing-hyp-coord-cvgce}, see Propositions \ref{prop:sing-hyp-coord-cvgce} and \ref{prop:sing-hyp-coord-cvgce-sacrifice}. In Section \ref{sec:ex-platypus} we discuss how to bound the term \(  \|D^2 \Phi_{\xi_0}(e^{(1)},\cdot)\| \) in specific coordinate systems, thus proving the second bound in Theorem~\ref{thm:main-phase}. In Section~\ref{sec-general} we give some a-priori bounds for the variation of hyperbolic coordinates, and in Section \ref{sec-slow}, we  take advantage of the quasi-hyperbolicity conditions to turn those abstract a-priori bounds into concrete bounds and so complete the proof of the first bound of Theorem~\ref{thm:main-phase}.

\section{Convergence of Hyperbolic Coordinates}\label{sec:hypcoords-cvgce}

In this section we prove Theorem \ref{thm:sing-hyp-coord-cvgce}. 
In Section \ref{subsec:a priori cvgce}) we prove a-priori bounds on $\|e^{(k)} - e^{(i)}\|$ and $\|e^{(k)}_i\|$ that do not assume any hyperbolicity at all apart form the existence of hyperboloic coordinates. We then use these estimates to find more explicit bounds assuming  conditions \eqref{eq:condition-I} and condition \eqref{eq:condition-II} in Definition \ref{def:nonsingular-QHv2-I}.
%\todelete{Because our \emph{a priori} assumptions in Lemma \ref{lem:hyp-coord-cvgce-apriori} are so general, there are many types of hyperbolic behavior that they may be applied to. Our primary interest is in strongly dissipative singular hyperbolic dynamical systems, such as the geometric Lorenz attractor and related systems. But by modifying our hyperbolicity assumptions, Lemma \ref{lem:hyp-coord-cvgce-apriori} may also be applied to weakly hyperbolic systems with varying degrees of dissipativity. To make our intermediate results more widely applicable, \toadd{we present bounds on the convergence of hyperbolic coordinates in two situations: when $\Phi$ is non-singular hyperbolic and its domain is compact (such as Anosov systems and the H\'enon map) and when $\Phi$ is singular hyperbolic and strongly dissipative (such as the geometric Lorenz model). }}

\subsection{\emph{A priori} bounds}\label{subsec:a priori cvgce}

We recall the definition of co-eccentiricity in \eqref{eq:eccentricity1} and let 
\begin{equation}\label{eq:Ctilde-def}
\widetilde C_{\xi_0, k} = \max_{1 \leq i \leq k} \sqrt{ \frac{2}{1-C_{\xi_0, i}^2}}.
\end{equation}
We note that the co-eccentricity of a sequence of linear maps, unlike the determinant, is not multiplicative. So there need be no relationship between product or sum of the pointwise single-step co-eccentricities $C_{\xi_j, 1}$ and the accumulated eccentricity $C_{\xi_0, k}$. In particular, $C_{\xi_0, k}$ need not be monotone in $k$. For the next two lemmas we just suppose that $\xi_0$ is a point at which hyperbolic coordinates of order $i$ are defined for all $1 \leq i \leq k$.

\begin{lemma}\label{lem:hyp-coord-cvgce-apriori}
For every $1 \leq i \leq k$
\begin{align}
\label{eq:hyp-coord-cvgce-apriori-1}
\|e^{(k)} - e^{(i)}\| 
& \leq \widetilde C_{\xi_0,k} \sum_{j=i}^{k-1} \frac{C_{\xi_0,j} \|D\Phi_{\xi_0}^j\|\|D\Phi_{\xi_j}\|}{\|D\Phi_{\xi_0}^{j+1}\|};
\\ 
\label{eq:hyp-coord-cvgce-apriori-2}
\|e^{(k)}_i\| 
& \leq \|(D\Phi_{\xi_0}^i)^{-1}\|^{-1} + \widetilde C_{\xi_0,k}\|D\Phi_{\xi_0}^i\|  \sum_{j=i}^{k-1} \frac{C_{\xi_0,j} \|D\Phi_{\xi_0}^j\|\|D\Phi_{\xi_j}\|}{\|D\Phi_{\xi_0}^{j+1}\|} ;
\\
\label{eq:hyp-coord-cvgce-apriori-3}
\frac{\|e^{(k)}_i\|}{|\det D\Phi_{\xi_0}^i|} 
&
\leq \frac{1}{\|D\Phi_{\xi_0}^i\|} + \widetilde C_{\xi_0,k}\|D\Phi_{\xi_0}^i\| \sum_{j=i}^{k-1} \frac{|\det D\Phi_{\xi_i}^{j-i}|  \|D\Phi_{\xi_j}\|}{\|D\Phi_{\xi_0}^j\| \|D\Phi_{\xi_0}^{j+1}\|}.
\end{align}

%
%\begin{equation}\label{eq:hyp-coord-cvgce-apriori-1}
%\|e^{(k)} - e^{(i)}\| \leq \widetilde C_{\xi_0,k} \sum_{j=i}^{k-1} \frac{C_{\xi_0,j} \|D\Phi_{\xi_0}^j\|\|D\Phi_{\xi_j}\|}{\|D\Phi_{\xi_0}^{j+1}\|};
%\end{equation}
%\begin{equation}\label{eq:hyp-coord-cvgce-apriori-2}
%\|e^{(k)}_i\| \leq \|(D\Phi_{\xi_0}^i)^{-1}\|^{-1} + \widetilde C_{\xi_0,k}\|D\Phi_{\xi_0}^i\|  \sum_{j=i}^{k-1} \frac{C_{\xi_0,j} \|D\Phi_{\xi_0}^j\|\|D\Phi_{\xi_j}\|}{\|D\Phi_{\xi_0}^{j+1}\|} ;
%\end{equation}
%\begin{equation}\label{eq:hyp-coord-cvgce-apriori-3}
%\frac{\|e^{(k)}_i\|}{|\det D\Phi_{\xi_0}^i|} \leq \frac{1}{\|D\Phi_{\xi_0}^i\|} + \widetilde C_{\xi_0,k}\|D\Phi_{\xi_0}^i\| \sum_{j=i}^{k-1} \frac{|\det D\Phi_{\xi_i}^{j-i}| \cdot \|D\Phi_{\xi_j}\|}{\|D\Phi_{\xi_0}^j\|\cdot \|D\Phi_{\xi_0}^{j+1}\|}.
%\end{equation}
\end{lemma}
A less sharp but more elegant set of bounds  can be obtained  by recalling that 
\begin{equation}\label{eq:tail-def}
C_{\xi_j,1} := \frac{\|(D\Phi_{\xi_j})^{-1}\|^{-1}}{\|D\Phi_{\xi_j}\|}
\quad\text{and lettiing} \quad 
\mathcal{T}_{\xi_0, i}^{(k)} := \sum_{j=i}^{k-1} \frac{C_{\xi_0, j}}{C_{\xi_j, 1}}.
\end{equation}
We can then show the following. 
\begin{lemma}\label{lem:sacrifice}
For every $1 \leq i \leq k$
\begin{align}
\label{eq:hyp-coord-cvgce-sacrifice-1}
\|e^{(k)} - e^{(i)}\| 
& \leq \mathcal{T}_{\xi_0,i}^{(k)}\widetilde C_{\xi_0,k} ;
\\
\label{eq:hyp-coord-cvgce-sacrifice-2}
\|e^{(k)}_i\| 
&
\leq \|(D\Phi_{\xi_0}^i)^{-1}\|^{-1} + \|D\Phi_{\xi_0}^i\| \mathcal{T}_{\xi_0,i}^{(k)}\widetilde C_{\xi_0,k} ;
\\
\label{eq:hyp-coord-cvgce-sacrifice-3}
\frac{\|e^{(k)}_i\|}{|\det D\Phi_{\xi_0}^i|} 
& 
\leq \frac{1}{\|D\Phi_{\xi_0}^i\|} + \widetilde C_{\xi_0,k}\|D\Phi_{\xi_0}^i\|  \sum_{j=i}^{k-1} \frac{|\det D\Phi_{\xi_i}^{j-i}|}{\|D\Phi_{\xi_0}^j\|^2 C_{\xi_j,1}}.
\end{align}
\end{lemma}

Each of these sets of estimates will be used when assuming either  Condition \eqref{eq:condition-I}  or \eqref{eq:condition-II}   in the Definition \ref{def:nonsingular-QHv2-I} of quasi-hyperbolicity. 
Assuming \eqref{eq:condition-I} it will be more convenient to apply  Lemma~\ref{lem:hyp-coord-cvgce-apriori} whereas assuming \eqref{eq:condition-II} it will be more convenient to use Lemma~\ref{lem:sacrifice}.

\begin{proof}[Proof of Lemma \ref{lem:hyp-coord-cvgce-apriori}]
    To estimate \( \|e^{(k)}-e^{(i)}\|\), we write $\|e^{(k)} - e^{(i)}\| \leq \sum_{j=i}^{k-1} \|e^{(j+1)} - e^{(j)}\|$ and estimate \( \|e^{(j+1)}-e^{(j)}\|\) for \(j\in\{i,...,k-1\}\). We 
  write \begin{equation}\label{eq:ej-theta}
  e^{(j)}=\cos\theta e^{(j+1)}+\sin\theta f^{(j+1)}
  \end{equation}
  for some $\theta = \theta_j$, $|\theta| \leq \pi/2$, which implies
  \begin{equation}\label{eq:ej-sin-est}
  \|e^{(j+1)} - e^{(j)}\| = \big((1-\cos\theta)^2 + \sin^2\theta\big)^{1/2} \leq \sqrt{2}|\sin\theta|.
  \end{equation}
  By orthogonality of $\{e^{(j+1)}_{j+1}, f^{(j+1)}_{j+1}\}$, after applying $D\Phi_{\xi_0}^j$ to both sides of \eqref{eq:ej-theta} and taking the norm, we get
\[
  \|e^{(j)}_{j+1}\|^2=\cos^2\theta \|e^{(j+1)}_{j+1}\|^2+\sin^2\theta \|f^{(j+1)}_{j+1}\|^2.
\]
This implies
\[
\sin^2\theta=\frac{\left(\frac{\|e^{(j)}_{j+1}\|}{\|f^{(j+1)}_{j+1}\|}\right)^2-\left(\frac{\|e^{(j+1)}_{j+1}\|}{\|f^{(j+1)}_{j+1}\|}\right)^2}{1-\left(\frac{\|e^{(j+1)}_{j+1}\|}{\|f^{(j+1)}_{j+1}\|}\right)^2}\leq
 \frac{\left(\frac{\|e^{(j)}_{j+1}\|}{\|f^{(j+1)}_{j+1}\|}\right)^2}{1-\left(\frac{\|e^{(j+1)}_{j+1}\|}{\|f^{(j+1)}_{j+1}\|}\right)^2}.
\]
%where the last inequality follows from the fact that \(\|e^{(j+1)}_{j+1}\|\leq\|e^{(j)}_{j+1}\|\).
Notice that \[\|e^{(j)}_{j+1}\|\leq \|D\Phi_{\xi_{j}}\|\|e^{(j)}_j\| = \|D\Phi_{\xi_j}\| \|(D\Phi_{\xi_0}^j)^{-1}\|^{-1} \quad \textrm{and} \quad \|f^{(j+1)}_{j+1}\| = \|D\Phi_{\xi_0}^{j+1}\|,\] and also that \[\frac{\|e^{(j+1)}_{j+1}\|}{\|f^{(j+1)}_{j+1}\|} = C_{\xi_{0},j+1}.\] 
Using also the fact that $\|(D\Phi_{\xi_0}^j)^{-1}\|^{-1} = C_{\xi_0,j} \|D\Phi_{\xi_0}^j\|$, we obtain
\begin{equation}\label{eq:sin}
\begin{aligned}
 \sin^2\theta &\leq
 \frac{\left(\frac{\|e^{(j)}_{j+1}\|}{\|f^{(j+1)}_{j+1}\|}\right)^2}{1-\left(\frac{\|e^{(j+1)}_{j+1}\|}{\|f^{(j+1)}_{j+1}\|}\right)^2} \leq 
  \frac{1}{1- C^{2}_{\xi_{0},j+1}}
 \frac{\|D\Phi_{\xi_{j}}\|^{2}\|(D\Phi_{\xi_0}^j)^{-1}\|^{-2}}{\|D\Phi_{\xi_0}^{j+1}\|^{2} } \\ 
 &= \frac{1}{1-C^2_{\xi_0, j+1}} \frac{C_{\xi_0, j}^2 \|D\Phi_{\xi_0}^j\|^2\|D\Phi_{\xi_j}\|^2}{\|D\Phi_{\xi_0}^{j+1}\|^2}
 \end{aligned}
 \end{equation}

Therefore, \eqref{eq:ej-sin-est} and \eqref{eq:sin} give us:
%\begin{align*}
%\|e^{(k)}-e^{(i)}\| &\leq\sum_{j=i}^{k-1}\|e^{(j)}-e^{(j+1)}\|
%\\
%&\leq \sqrt{\frac{2}{1- C_{\xi_0, i}^2}} \sum_{j=i}^{k-1} \frac{C_{\xi_0,j} \|D\Phi_{\xi_0}^j\|\|D\Phi_{\xi_j}\|}{\|D\Phi_{\xi_0}^{j+1}\|} \\
%&\leq \widetilde C_{\xi_0,k}\sum_{j=i}^{k-1} \frac{C_{\xi_0,j} \|D\Phi_{\xi_0}^j\|\|D\Phi_{\xi_j}\|}{\|D\Phi_{\xi_0}^{j+1}\|}.
%\end{align*}
\[
\begin{aligned}
\|e^{(k)}-e^{(i)}\| &\leq\sum_{j=i}^{k-1}\|e^{(j)}-e^{(j+1)}\|\\
&\leq \sqrt{\frac{2}{1- C_{\xi_0, i}^2}} \sum_{j=i}^{k-1} \frac{C_{\xi_0,j} \|D\Phi_{\xi_0}^j\|\|D\Phi_{\xi_j}\|}{\|D\Phi_{\xi_0}^{j+1}\|} \\
&\leq \widetilde C_{\xi_0,k}\sum_{j=i}^{k-1} \frac{C_{\xi_0,j} \|D\Phi_{\xi_0}^j\|\|D\Phi_{\xi_j}\|}{\|D\Phi_{\xi_0}^{j+1}\|}.
\end{aligned}
\]
This gives us \eqref{eq:hyp-coord-cvgce-apriori-1}. To prove \eqref{eq:hyp-coord-cvgce-apriori-2}, we use \eqref{eq:hyp-coord-cvgce-apriori-1} to show: 
%\begin{align*}
%    \|e^{(k)}_i\| &\leq \|e^{(i)}_i\| + \|D\Phi_{\xi_0}^i\| \|e^{(k)} - e^{(i)}\| \\
%    &\leq \|(D\Phi_{\xi_0}^i)^{-1}\|^{-1} + \widetilde C_{\xi_0,k} \|D\Phi_{\xi_0}^i\| \sum_{j=i}^{k-1} \frac{C_{\xi_0,j} \|D\Phi_{\xi_0}^j\|\|D\Phi_{\xi_j}\|}{\|D\Phi_{\xi_0}^{j+1}\|}.
%\end{align*}
\[
    \|e^{(k)}_i\| \leq \|e^{(i)}_i\| + \|D\Phi_{\xi_0}^i\| \|e^{(k)} - e^{(i)}\| 
    \leq \|(D\Phi_{\xi_0}^i)^{-1}\|^{-1} + \widetilde C_{\xi_0,k} \|D\Phi_{\xi_0}^i\| \sum_{j=i}^{k-1} \frac{C_{\xi_0,j} \|D\Phi_{\xi_0}^j\|\|D\Phi_{\xi_j}\|}{\|D\Phi_{\xi_0}^{j+1}\|}.
\]
Finally, noting $C_{\xi_0,j}\|D\Phi_{\xi_0}^j\|/|\det D\Phi_{\xi_0}^j| = \|D\Phi_{\xi_0}^j\|^{-1}$, factoring out $|\det D\Phi_{\xi_0}^i|$ from the summands in \eqref{eq:hyp-coord-cvgce-apriori-2} gives us:
%\begin{align*}
%\frac{C_{\xi_0,j} \|D\Phi_{\xi_0}^j\|\|D\Phi_{\xi_j}\|}{\|D\Phi_{\xi_0}^{j+1}\|} &= \frac{|\det D\Phi_{\xi_0}^i||\det D\Phi_{\xi_i}^{j-i}|}{|\det D\Phi_{\xi_0}^j|} \times \frac{C_{\xi_0,j}\|D\Phi_{\xi_0}^j\|\|D\Phi_{\xi_j}\|}{\|D\Phi_{\xi_0}^{j+1}\|} \\
%&= |\det D\Phi_{\xi_0}^i| \frac{|\det D\Phi_{\xi_i}^{j-i}|\|D\Phi_{\xi_j}\|}{\|D\Phi_{\xi_0}^j\|\|D\Phi_{\xi_0}^{j+1}\|},
%\end{align*}
\[
\frac{C_{\xi_0,j} \|D\Phi_{\xi_0}^j\|\|D\Phi_{\xi_j}\|}{\|D\Phi_{\xi_0}^{j+1}\|} = \frac{|\det D\Phi_{\xi_0}^i||\det D\Phi_{\xi_i}^{j-i}|}{|\det D\Phi_{\xi_0}^j|} \frac{C_{\xi_0,j}\|D\Phi_{\xi_0}^j\|\|D\Phi_{\xi_j}\|}{\|D\Phi_{\xi_0}^{j+1}\|} = |\det D\Phi_{\xi_0}^i| \frac{|\det D\Phi_{\xi_i}^{j-i}|\|D\Phi_{\xi_j}\|}{\|D\Phi_{\xi_0}^j\|\|D\Phi_{\xi_0}^{j+1}\|},
\]
and since $\|(D\Phi_{\xi_0}^i)^{-1}\|^{-1}/|\det D\Phi_{\xi_0}^i| = \|D\Phi_{\xi_0}\|^{-1}$, dividing \eqref{eq:hyp-coord-cvgce-apriori-2} by $|\det D\Phi_{\xi_0}^i|$ gives us \eqref{eq:hyp-coord-cvgce-apriori-3}. 
\end{proof}

\begin{proof}[Proof of Lemma \ref{lem:sacrifice}]
Observe first of all  that we always have $\|D\Phi_{\xi_0}^{j+1}\| \geq \|D\Phi_{\xi_0}^j\| \|(D\Phi_{\xi_j})^{-1}\|^{-1}$ and so 
\begin{equation}\label{eq:sacrifice-1}
\frac{\|D\Phi_{\xi_0}^j\|\|D\Phi_{\xi_j}\|}{\|D\Phi_{\xi_0}^{j+1}\|} \leq \frac{1}{C_{\xi_j,1}}
\quand
\frac{\|D\Phi_{\xi_j}\|}{\|D\Phi_{\xi_0}^j\|\|D\Phi_{\xi_0}^{j+1}\|} \leq \frac{1}{\|D\Phi_{\xi_0}^j\|^2C_{\xi_j,1}}
\end{equation}
Substituting the first inequality into  \eqref{eq:hyp-coord-cvgce-apriori-1} and \eqref{eq:hyp-coord-cvgce-apriori-2} gives \eqref{eq:hyp-coord-cvgce-sacrifice-1} and \eqref{eq:hyp-coord-cvgce-sacrifice-2} respectively, and substituting the second inequality into   \eqref{eq:hyp-coord-cvgce-apriori-3}  gives \eqref{eq:hyp-coord-cvgce-sacrifice-3}.  \end{proof}

\begin{remark}
    The first step in the proof of Lemma \ref{lem:hyp-coord-cvgce-apriori} is to use the triangle inequality to write $\|e^{(k)} - e^{(i)}\| \leq \sum_{j=i}^{k-1} \|e^{(j+1)} - e^{(j)}\|$, and then to estimate each term $\|e^{(j+1)} - e^{(j)}\|$. Strictly speaking, this first step is not necessary to obtain an \emph{a priori} bound; one could use the same arguments to directly estimate $\|e^{(k)} - e^{(i)}\|$ instead. Doing so would, for example, give us the bound 
    \begin{equation}\label{eq:cvgce-alternative-1}
        \|e^{(k)} - e^{(i)}\| \leq  \widetilde C_{\xi_0,k}\frac{ C_{\xi_0,i} \|D\Phi_{\xi_0}^i\|\|D\Phi_{\xi_i}^{k-i}\|}{\|D\Phi_{\xi_i}^k\|}
    \end{equation}
%    \begin{equation}\label{eq:cvgce-alternative-2}
%\|e^{(k)}_i\| \leq \|(D\Phi_{\xi_0})^{-1}\|^{-1} + \frac{C_{\xi_0,k}\|D\Phi_{\xi_i}^{k-i}\|\|D\Phi_{\xi_0}^i\|^2}{\|D\Phi_{\xi_0}^k\|}\sqrt{\frac{2}{1-\widetilde C_{\xi_0,k}^2}};
%    \end{equation} 
%    \begin{equation}\label{eq:cvgce-alternative-3}
%        \frac{\|e^{(k)}_i\|}{|\det D\Phi_{\xi_0}^i|} \leq \frac{1}{\|D\Phi_{\xi_0}^i\|} + \frac{\|D\Phi_{\xi_i}^{k-i}\|}{\|D\Phi_{\xi_0}^k\|}\sqrt{\frac{2}{1-\widetilde C_{\xi_0,k}^2}}.
%    \end{equation}
    
    instead of the bound in \eqref{eq:hyp-coord-cvgce-apriori-1}, and similar alternative bounds to \eqref{eq:hyp-coord-cvgce-apriori-2} and \eqref{eq:hyp-coord-cvgce-apriori-3} can also be derived. Rather than in terms of the sum
    \begin{equation}\label{eq:comparison-j}
    \sum_{j=i}^{k-1}\frac{C_{\xi_0,j}\|D\Phi_{\xi_0}^j\|\|D\Phi_{\xi_j}\|}{\|D\Phi_{\xi_0}^{j+1}\|},
    \end{equation}
    we instead get a bound in terms of the quotient
    \begin{equation}\label{eq:comparison-k}
    \frac{C_{\xi_0,i}\|D\Phi_{\xi_0}^i\|\|D\Phi_{\xi_i}^{k-i}\|}{\|D\Phi_{\xi_0}^k\|}.
    \end{equation}
    The former takes the sum of the estimates for each transition from $\xi_j$ to $\xi_{j+1}$ under $\Phi$, whereas the latter estimates directly the transition from $\xi_i$ to $\xi_k$ under $\Phi^{k-i}$. The main reason we estimate the sum in \eqref{eq:comparison-j} instead of the term in \eqref{eq:comparison-k} becomes apparent only after introducing quasi-hyperbolicity assumptions. We will see in Propositions \ref{prop:sing-hyp-coord-cvgce} and \ref{prop:sing-hyp-coord-cvgce-sacrifice} that $\|e^{(k)} - e^{(i)}\|$ and $\|e^{(k)}_i\|$ both have upper bounds that are exponential in $i$ but independent of $k$. This is because we will approximate the sum in \eqref{eq:comparison-j} with the tail of a geometric series, which decays exponentially with $i$ and is independent of $k$. However, the expression in \eqref{eq:comparison-k} cannot be given an upper bound independent of $k$ without introducing much stronger restrictions than quasi-hyperbolicity.
\end{remark}

\subsection{Convergence with hyperbolicity assumptions}\label{subsec:cvgce}

We now estimate $\|e^{(k)} - e^{(i)}\|$ and $\|e^{(k)}_i\|$ applying the bounds given in the Definition \ref{def:nonsingular-QHv2-I}  of quasi-hyperbolicity, treating separately the situations in which Condition \eqref{eq:condition-I} and Condition \eqref{eq:condition-II}  are satisfied (sections \ref{subsubsec:cvgce-I} and \ref{subsubsec:cvgce-II} respectively). Note that the \emph{a priori} estimates in Lemma \ref{lem:hyp-coord-cvgce-apriori} are stronger than the estimates in Lemma~\ref{lem:sacrifice}. However, the conclusions we obtain from these lemmas, which are formulated in Propositions \ref{prop:sing-hyp-coord-cvgce} and \ref{prop:sing-hyp-coord-cvgce-sacrifice} respectively, are not similarly related: it is not immediate that one set of estimates is stronger than the other. This is because we bound different terms in different ways in the two cases.

\subsubsection{Convergence with quasi-hyperbolicity of type (I)}\label{subsubsec:cvgce-I}

First we suppose that the constants satisfy  Condition \eqref{eq:condition-I} of Definition \ref{def:nonsingular-QHv2-I} . 

%\todelete{In particular, since $c < \lambda^3/\Gamma^3\widetilde\Gamma^3$ by~\eqref{eq:condition-I}, and since $\widetilde\Gamma \geq \Gamma \geq \lambda$, it follows that $c < \lambda/\Gamma \widetilde \Gamma$. Additionally, we know $\lambda^2 > \widetilde \Gamma b$ from \eqref{eq:condition-I}. }

\begin{proposition}\label{prop:sing-hyp-coord-cvgce}
    Suppose $\xi_0$ is singular quasi-hyperbolic up to time $k$ and satisfies condition \eqref{eq:condition-I}. Then: 
    \begin{equation}\label{eq:cvgce-k-quasi-hyp-1}
\|e^{(k)} - e^{(i)}\| \leq Q_1 \left( \frac{\Gamma\widetilde\Gamma c}{\lambda}\right)^i;
\end{equation}
\begin{equation}\label{eq:cvgce-k-quasi-hyp-2}
\|e^{(k)}_i\| \leq Q_1\left( \frac{\Gamma^2 \widetilde \Gamma c}{\lambda}\right)^i;
\end{equation}
\begin{equation}\label{eq:cvgce-k-quasi-hyp-3}
\frac{\|e^{(k)}_i\|}{|\det D\Phi_{\xi_0}^i|} \leq Q_2 \left( \frac{\Gamma \widetilde \Gamma}{\lambda^2}\right)^i
\end{equation}

\end{proposition}

\begin{proof}
We first note that by the second set of inequalities in \eqref{eq:constants-basic-sing1} and the fact that $c < 1$,
    \begin{equation}\label{eq:sqrt-est}
\widetilde C_{\xi_0, k} = \max_{1 \leq i \leq k} \sqrt{\frac{2}{1-C_{\xi_0,i}^2}} \leq \max_{1 \leq i \leq k} \sqrt{ \frac{2}{1-B^2 c^{2i}}} = \sqrt{\frac{2}{1-B^2 c^2}} = Q_0. 
\end{equation}
    So \eqref{eq:constants-basic-sing1}, \eqref{eq:constants-basic2-sing1}, \eqref{eq:hyp-coord-cvgce-apriori-1}, and \eqref{eq:sqrt-est} give us: 
    \begin{equation}\label{eq:sing-sum-est-v1}
    \begin{aligned}
        \|e^{(k)} - e^{(i)}\| &\leq \widetilde C_{\xi_0,k} \sum_{j=i}^{k-1} \frac{C_{\xi_0, j} \|D\Phi_{\xi_0}^j\|\|D\Phi_{\xi_j}\|}{\|D\Phi_{\xi_0}^{j+1}\|}\\
        &\leq Q_0\sum_{j=i}^{k-1} \frac{BD^2c^j \Gamma^{j+1}\widetilde\Gamma^j}{C\lambda^{j+1}} \\
        &\leq \frac{Q_0BD^2\Gamma}{C\lambda} \sum_{j=i}^{\infty} \left( \frac{\Gamma\widetilde\Gamma c}{\lambda}\right)^j \\
        &= \frac{Q_0 BD^2\Gamma}{C\lambda} \frac{1}{1-\frac{\Gamma \widetilde \Gamma c}{\lambda}} \left( \frac{\Gamma \widetilde \Gamma c}{\lambda}\right)^i \\
        &= \frac{Q_0 BD^2\Gamma}{C(\lambda - \Gamma \widetilde \Gamma c)} \left( \frac{\Gamma \widetilde \Gamma c}{\lambda}\right)^i.
        \end{aligned}
    \end{equation}
    The first equality follows because $\Gamma \widetilde \Gamma c < \lambda$ as a consequence of \eqref{eq:condition-I}. So \eqref{eq:cvgce-k-quasi-hyp-1} now follows. 
    Next, note:
    \begin{equation}\label{eq:contraction-est}
    \|(D\Phi_{\xi_0}^i)^{-1}\|^{-1} = \|D\Phi_{\xi_0}^i\|C_{\xi_0,i} \leq BD(\Gamma c)^i.
    \end{equation}
    So from \eqref{eq:hyp-coord-cvgce-apriori-2}, \eqref{eq:contraction-est}, and \eqref{eq:sing-sum-est-v1}, we obtain: 
    \begin{align*}
        \|e^{(k)}_i\| &\leq \|(D\Phi_{\xi_0}^i)^{-1}\|^{-1} + \widetilde C_{\xi_0,k}\|D\Phi_{\xi_0}^i\|  \sum_{j=i}^{k-1} \frac{C_{\xi_0,j} \|D\Phi_{\xi_0}^j\|\|D\Phi_{\xi_j}\|}{\|D\Phi_{\xi_0}^{j+1}\|} \\
        &\leq BD(\Gamma c)^i + \frac{Q_0 BD^3\Gamma}{C(\lambda - \Gamma \widetilde \Gamma c)} \left( \frac{\Gamma^2 \widetilde \Gamma c}{\lambda}\right)^i \\
        &\leq \left( BD + \frac{Q_0 BD^3\Gamma}{C(\lambda - \Gamma \widetilde \Gamma c)}\right) \left( \frac{\Gamma^2 \widetilde \Gamma c}{\lambda}\right)^i \\
        &= Q_1 \left( \frac{\Gamma^2 \widetilde \Gamma c}{\lambda}\right)^i.
    \end{align*}
    Note the final inequality follows because $\Gamma c \leq \Gamma^2 \widetilde \Gamma c /\lambda$ because $\Gamma \geq \lambda$ and $\widetilde \Gamma \geq 1$ by \eqref{eq:condition-I}. Finally, applying \eqref{eq:constants-basic-sing1}, \eqref{eq:constants-basic2-sing1}, and \eqref{eq:sqrt-est} to \eqref{eq:hyp-coord-cvgce-apriori-3}, we obtain: 
    \begin{align*}
        \frac{\|e^{(k)}_i\|}{|\det D\Phi_{\xi_0}^i|} &\leq \frac{1}{\|D\Phi_{\xi_0}^i\|} + \widetilde C_{\xi_0,k}\|D\Phi_{\xi_0}^i\| \sum_{j=i}^{k-1} \frac{|\det D\Phi_{\xi_i}^{j-i}|  \|D\Phi_{\xi_j}\|}{\|D\Phi_{\xi_0}^j\| \|D\Phi_{\xi_0}^{j+1}\|} \\
        &\leq \frac 1{C\lambda^i} + Q_0D\Gamma^i \sum_{j=i}^{k-1} \frac{b^{j-i} D\Gamma \widetilde \Gamma^j}{C^2 \lambda^{2j+1}} \\
        &\leq \frac{1}{C\lambda^i} + \frac{Q_0D^2\Gamma}{C^2\lambda} \frac{\Gamma^i}{b^i} \sum_{j=i}^\infty \left( \frac{\widetilde \Gamma b}{\lambda^2}\right)^j \\
        &= \frac{1}{C\lambda^i} + \frac{Q_0 D^2\Gamma}{C^2\lambda} \frac{1}{1-\frac{\widetilde \Gamma b}{\lambda^2}} \frac{\Gamma^i}{b^i} \left( \frac{\widetilde \Gamma b}{\lambda^2}\right)^i \\
        &= \frac 1{C\lambda^i} + \frac{Q_0 D^2\Gamma\lambda }{C^2(\lambda^2 - \widetilde \Gamma b)} \left( \frac{\Gamma \widetilde \Gamma}{\lambda^2}\right)^i \\
        &\leq \left( \frac 1 C + \frac{Q_0 D^2\Gamma\lambda }{C^2(\lambda^2 - \widetilde \Gamma b)} \right) \left( \frac{\Gamma \widetilde \Gamma}{\lambda^2}\right)^i \\
        &= Q_2 \left( \frac{\Gamma \widetilde \Gamma}{\lambda^2}\right)^i.
    \end{align*}
The final inequality holds because $\Gamma \widetilde \Gamma/\lambda^2 \geq 1/\lambda$, since $\Gamma \geq \lambda$ and $\widetilde \Gamma \geq 1$. 
\end{proof}

\subsubsection{Convergence with quasi-hyperbolicity of type (II)}\label{subsubsec:cvgce-II}

We now suppose the constants satisfy condition \eqref{eq:condition-II}  in Definition \ref{def:nonsingular-QHv2-I}.
  
\begin{proposition}\label{prop:sing-hyp-coord-cvgce-sacrifice}
Suppose $\xi_0$ is singular hyperbolic up to time $k$ and satisfies \eqref{eq:condition-II}. Then:
\begin{equation}\label{eq:cvgce-k-quasi-hyp-sac-1}
\|e^{(k)} - e^{(i)}\| \leq \widetilde Q_1\left( \frac{c}{\tilde c}\right)^i;
\end{equation}
\begin{equation}\label{eq:cvgce-k-quasi-hyp-sac-2}
\|e^{(k)}_i\| \leq \widetilde Q_1 \left( \frac{\Gamma c}{\tilde c}\right)^i;
\end{equation}
\begin{equation}\label{eq:cvgce-k-quasi-hyp-sac-3}
\frac{\|e^{(k)}_i\|}{|\det D\Phi_{\xi_0}^i|} \leq \widetilde Q_2 \left( \frac{\Gamma}{\lambda^2 \tilde c}\right)^i.
\end{equation}
\end{proposition}

\begin{proof}[Proof of Proposition \ref{prop:sing-hyp-coord-cvgce-sacrifice}]

The proof essentially consists of applying the estimates in Definition \ref{def:nonsingular-QHv2-I} and in \eqref{eq:condition-II} to the \emph{a priori} estimates in Lemma \ref{lem:sacrifice}.
    Observe first that since $c < \tilde c$, 
    \begin{equation}\label{eq:sing-T-est}
    \mathcal T^{(k)}_{\xi_0, i} = \sum_{j=i}^{k-1} \frac{C_{\xi_0,j}}{C_{\xi_j,1}} < \sum_{j=i}^{\infty} \frac{Bc^j}{\widetilde B\tilde c^j} = \frac{B}{\widetilde B} \frac{1}{1-\frac{c}{\tilde c}} \left( \frac c{\tilde c}\right)^i = \frac{B\tilde c}{\widetilde B(\tilde c - c)} \left( \frac c{\tilde c}\right)^i.
    \end{equation}
    Note \eqref{eq:sqrt-est} also applies in this setting. By \eqref{eq:sqrt-est} and \eqref{eq:sing-T-est}, we get: 
    \begin{equation}\label{eq:sing-tail-cvgce-est}
        \mathcal T^{(k)}_{\xi_0, i} \widetilde C_{\xi_0,k} \leq \frac{Q_0B\tilde c}{\widetilde B(\tilde c - c)} \left( \frac{c}{\tilde c}\right)^i.
    \end{equation}
    Now \eqref{eq:cvgce-k-quasi-hyp-sac-1} follows from \eqref{eq:hyp-coord-cvgce-sacrifice-1} and \eqref{eq:sing-tail-cvgce-est}. Meanwhile, by \eqref{eq:hyp-coord-cvgce-sacrifice-2} and \eqref{eq:sing-tail-cvgce-est}, using again that $\|(D\Phi_{\xi_0}^i)^{-1}\|^{-1} \leq BD(\Gamma c)^i$ by \eqref{eq:contraction-est}, we get: 
    \begin{align*}
\|e^{(k)}_i\| &\leq \|(D\Phi_{\xi_0}^i)^{-1}\|^{-1} + \|D\Phi_{\xi_0}^i\|\mathcal T_{\xi_0,i}^{(k)} \widetilde C_{\xi_0,k} \\
&\leq BD(\Gamma c)^{i} + \frac{Q_0B\tilde c}{\widetilde B(\tilde c - c)} \left( \frac{\Gamma c}{\tilde c}\right)^i \\
&\leq \left( BD + \frac{Q_0B\tilde c}{\widetilde B(\tilde c - c)}\right) \left( \frac{\Gamma c}{\tilde c}\right)^i \\
&= \widetilde Q_1 \left( \frac{\Gamma c}{\tilde c}\right)^i.
    \end{align*}
    The final inequality holds because $\tilde c < 1$. This proves \eqref{eq:cvgce-k-quasi-hyp-sac-2}. Finally, if \eqref{eq:condition-II} holds, then by applying \eqref{eq:constants-basic-sing1} and \eqref{eq:constants-basic2-sing1} to \eqref{eq:hyp-coord-cvgce-sacrifice-3}, we obtain: 
    \begin{align*}
        \frac{\|e^{(k)}_i\|}{|\det D\Phi_{\xi_0}^i|} &\leq \frac{1}{\|D\Phi_{\xi_0}^i\|} + \|D\Phi_{\xi_0}^i\| \widetilde C_{\xi_0,k} \sum_{j=i}^{k-1} \frac{|\det D\Phi_{\xi_i}^{j-i}|}{\|D\Phi_{\xi_0}^j\|^2 C_{\xi_j,1}} \\
        &\leq \frac 1{C\lambda^i} + Q_0 D\Gamma^i \sum_{j=i}^{k-1} \frac{b^{j-i}}{C^2\lambda^{2j} \widetilde B \tilde c^j} \\
        &\leq \frac{1}{C\lambda^i} + \frac{Q_0D}{C^2\widetilde B} \frac{\Gamma^i}{b^i} \sum_{j=i}^\infty \left( \frac{b}{\lambda^2 \tilde c}\right)^j \\
        &= \frac{1}{C\lambda^i} + \frac{Q_0D}{C^2 \widetilde B} \frac{1}{1-\frac b{\lambda^2 \tilde c}} \frac{\Gamma^i}{b^i} \left( \frac{b}{\lambda^2 \tilde c}\right)^i \\
        &= \frac{1}{C\lambda^i} + \frac{Q_0 \lambda^2\tilde c}{C^2\widetilde B(\lambda^2 \tilde c - b)} \left( \frac{\Gamma}{\lambda^2 \tilde c}\right)^i \\
        &\leq \left( \frac 1 C + \frac{Q_0 D\lambda^2\tilde c}{C^2\widetilde B(\lambda^2 \tilde c - b)} \right) \left( \frac \Gamma{\lambda^2 \tilde c}\right)^i \\
        &= \widetilde Q_2 \left( \frac \Gamma{\lambda^2 \tilde c}\right)^i.
    \end{align*}
 The first equality holds by the second inequality in \eqref{eq:condition-II}, and the final inequality holds since $\Gamma/\lambda^2 \tilde c \geq 1/\lambda$, because $\tilde c \leq 1$ and because $\Gamma \geq \lambda$ by \eqref{eq:constants-basic-x}. 
\end{proof}

\begin{proof}[Proof of Theorem \ref{thm:sing-hyp-coord-cvgce}]
Theorem \ref{thm:sing-hyp-coord-cvgce} follows immediately from \eqref{eq:cvgce-k-quasi-hyp-1} and \eqref{eq:cvgce-k-quasi-hyp-sac-1}.
\end{proof}

\section{Slow Variation in Local Coordinates}\label{sec:ex-platypus}
%\comment{Maybe this title is clearer and more consistent with the others?}

We are now ready to start the proof of Theorem \ref{thm:main-phase}. In this section we prove  the second inequality in \eqref{eq:main-phase}, which gives a more explicit bound for a given choice of local coordinates. This is relatively simple and quite general and contains a couple of bounds which we will use again in the following sections.

%\todelete{We are now ready to complete the proof of Theorem \ref{thm:main-phase}. As we have already proved the first inequality in \eqref{eq:main-phase}, the second inequality follows immediately from the following result. 
%}

\begin{proposition}\label{lem:ex-platypus} In normal coordinates based at $\xi_0$, the norm $\|D^{2}\Phi_{\xi_{0}}\|$ satisfies:
\begin{equation}\label{eq:second-derivative-norm}
\max_{\varsigma = x,y} \{ \|\partial_{\varsigma}(D\Phi_{\xi_{0}})\|\} \leq \|D^{2}\Phi_{\xi_{0}}\| \leq \sqrt2\max_{\varsigma = x,y} \{ \|\partial_{\varsigma}(D\Phi_{\xi_{0}})\|\}.
\end{equation} 
Moreover, for any tangent vector $v$ at $\xi_{0}$:
  \begin{equation}\label{eq:ex-platypus}
        \max_{\varsigma = x,y} \left\{ \left\|\left[ \partial_\varsigma\left( D\Phi_{\xi_0}\right)\right] v\right\|\right\} \leq \left\|D^2 \Phi_{\xi_0} \left(v, \cdot\right)\right\| \leq \sqrt 2 \max_{\varsigma = x,y} \left\{ \left\|\left[ \partial_\varsigma\left( D\Phi_{\xi_0}\right)\right] v\right\|\right\}.
    \end{equation}
 \end{proposition}
  Indeed, letting \( v= e^{(1)}\) and substituting the second inequality of   \eqref{eq:ex-platypus} into \eqref{eq:main-phase} we get the required estimate.

We will prove a more general version of Proposition \ref{lem:ex-platypus} in higher dimensions. Specifically, let $\Phi : M \to M$ be a $C^{2}$ map of a Riemannian $n$-manifold $M$, and let $(x^1, \ldots, x^n)$ be Riemannian normal coordiantes at $\xi_0$. For each $k = 1, \ldots, n$, we define the second-order partial derivative with respect to $x^{k}$ to be the linear map $\partial_{x^{k}}(D\Phi_{\xi_{0}}) : T_{\xi_{0}} M \to T_{\Phi(\xi_{0})}M$ given by:
\begin{equation}\label{eq:partial-dphi-coords}
\partial_{x^k}(D\Phi_{\xi_{0}}) = \begin{pmatrix}
    \displaystyle \frac{\partial^2 \Phi_1}{\partial x^1 \partial x^k}(\xi_{0}) & \cdots & \displaystyle \frac{\partial^2 \Phi_1}{\partial x^n \partial x^k}(\xi_{0}) \\
    \vdots & & \vdots \\
    \displaystyle \frac{\partial^2 \Phi_n}{\partial x^1 \partial x^k}(\xi_{0}) & \cdots & \displaystyle \frac{\partial^2 \Phi_n}{\partial x^n \partial x^k} (\xi_{0})
\end{pmatrix}.
\end{equation}

\begin{proposition}\label{lem:V-into-second-deriv-approx}
    For all $v \in T_{\xi_{0}}M$, in normal coordiantes at $\xi_0$,
 \begin{equation}\label{eq:D2-Phi-v-approx}
    \max_{1 \leq k \leq n} \| [\partial_{x^k}(D\Phi_{\xi_{0}})] v\| \leq \|D^2\Phi_{\xi_{0}}(v,\cdot)\| \leq \sqrt n \max_{1 \leq k \leq n} \|[\partial_{x^k}(D\Phi_{\xi_{0}})] v\|.
    \end{equation}
In particular, 
    \begin{equation}\label{eq:D^2-Phi-approx}
    \max_{1 \leq k \leq n} \|\partial_{x^k}(D\Phi_{\xi_{0}})\| \leq \|D^2\Phi_{\xi_{0}}\| \leq \sqrt n \max_{1 \leq k \leq n} \|\partial_{x^k}(D\Phi_{\xi_{0}})\|.
\end{equation}
\end{proposition}

This clearly implies Proposition \ref{lem:ex-platypus} when $n=2$.

To prove Proposition
\ref{lem:V-into-second-deriv-approx} we will prove a series of lemmas, most of which are linear-algebraic observations. Notice the middle term $D^2\Phi_{\xi_0}$ is in fact a bilinear map $T_{\xi_0}M \times T_{\xi_0}M \to T_{\Phi(\xi_0)}M$. We begin by introducing notation to study the coordinates of $D^2\Phi_{\xi_0}$ and show that the operator $\partial_{x^k} (D\Phi_{\xi_0})$ in \eqref{eq:partial-dphi-coords} in fact the \emph{monolinear} map $D^2\Phi_{\xi_0}(\cdot, \partial_{x^k})$.

Suppose $b : \mathbb R^n \times \mathbb R^n \to \mathbb R$ is a bilinear form. Given a basis $\mathcal B = \{\mathbf u_1, \ldots, \mathbf u_n\}$ of $\mathbb R^n$, let $\mathcal B^* = \{\omega^1, \ldots, \omega^n\}$ be the corresponding cobasis of linear functionals on $\mathbb R^n$, defined by $\omega^i(\mathbf u_j) = \delta_{ij}$ for all $1 \leq i,j \leq n$. Then the bilinear form $b$ can be written as
\[
b = \sum_{1 \leq i,j \leq k} b_{ij} \omega^i \otimes \omega^j,
\]
for some real coefficients $b_{ij} \in \mathbb R$. Consider now a bilinear map $B : \mathbb R^n \times \mathbb R^n \to \mathbb R^n$. In any basis $\mathcal B = (\mathbf u_1, \ldots, \mathbf u_n)$ of $\mathbb R^n$, writing $B = (B^1, \ldots, B^n)$, each $B^k$ is a bilinear form. Thus, again letting $\mathcal B^* = (\omega^1, \ldots, \omega^n)$ be the cobasis of $\mathcal B$, we can express $B$ as: 
\begin{equation}\label{eq:bilinear-matrix}
B = \begin{pmatrix}
    B^1 \\ \vdots \\ B^n
\end{pmatrix} = \begin{pmatrix}
    \sum_{i,j} B_{ij}^1 \omega^i \otimes \omega^j \\ \vdots \\
    \sum_{i,j} B_{ij}^n \omega^i \otimes \omega^j
\end{pmatrix}.
\end{equation}
For a particular $\mathbf v \in \mathbb R^n$, the map $B(\mathbf v, \cdot)$ given by $\mathbf x \mapsto B(\mathbf v, \mathbf x)$ is a linear transformation, and thus has a matrix representation in terms of any basis.  Referring to \eqref{eq:bilinear-matrix} and writing a vector $\mathbf v$ in coordinates as $\mathbf v = (v^1, \ldots, v^n)$, one sees that:
\[
B(\mathbf v, \cdot) = \begin{pmatrix}
    B^1(\mathbf v, \cdot) \\ \vdots \\ B^n(\mathbf v, \cdot)
\end{pmatrix}
= \begin{pmatrix}
    \sum_{i,j} B^1_{ij} v^i \omega^j \\ \vdots \\ \sum_{i,j}B^n_{ij} v^i \omega j
\end{pmatrix},
\]
and after expanding this ``covector form'' of $B(\mathbf v, \cdot)$, one sees that the matrix form of $B(\mathbf v, \cdot)$ is:
\begin{equation}\label{eq:vector-into-bilinear}
B(\mathbf v, \cdot) = \begin{pmatrix}
    \sum_{i=1}^n B^1_{i,1} v^i & \cdots & \sum_{i=1}^n B^1_{i,n} v^i \\
    \vdots & & \vdots \\
    \sum_{i=1}^n B^n_{i,1} v^i & \cdots & \sum_{i=1}^n B^n_{i,n} v^i
\end{pmatrix}.
\end{equation}
We will use this covector form to study $D^2 \Phi$.

\begin{lemma}\label{lem:D2V}
Let $(x^1, \ldots, x^n)$ be coordinates at $\xi_0$, and let $V = V^{1} \partial_{x^{1}} + \cdots + V^{n} \partial_{x^{n}}$, $V^{i} \in \mathbb R$, be a tangent vector at $\xi_{0}$. Then:
\begin{equation}\label{eq:D2V}
D^2\Phi(V,\cdot) = \begin{pmatrix}
    \displaystyle \sum_{i=1}^n \frac{\partial^2 \Phi_1}{\partial x^i \partial x^1} V^i & \cdots & \displaystyle \sum_{i=1}^n \frac{\partial^2 \Phi_1}{\partial x^i \partial x^n} V^i \\
    \vdots & & \vdots \\
    \displaystyle \sum_{i=1}^n \frac{\partial^2 \Phi_n}{\partial x^i \partial x^1} V^i & \cdots & \displaystyle \sum_{i=1}^n \frac{\partial^2 \Phi_n}{\partial x^i \partial x^n} V^i
\end{pmatrix}.
\end{equation}
In particular, for any of the coordinate vectors $\partial_{x^{k}}$:
\begin{equation}\label{eq:column-second-deriv-coords}
D^2\Phi(V, \partial_{x^k}) = \begin{pmatrix} \displaystyle \sum_{i=1}^n \frac{\partial^2 \Phi_1}{\partial x^i \partial x^k} V^i \\ \vdots \\ \displaystyle \sum_{i=1}^n \frac{\partial^2 \Phi_n}{\partial x^i \partial x^k} V^i 
\end{pmatrix} = \begin{pmatrix}
    \displaystyle \frac{\partial^2 \Phi_1}{\partial x^1 \partial x^k} & \cdots & \displaystyle \frac{\partial^2 \Phi_1}{\partial x^n \partial x^k} \\
    \vdots & & \vdots \\
    \displaystyle \frac{\partial^2 \Phi_n}{\partial x^1 \partial x^k} & \cdots & \displaystyle \frac{\partial^2 \Phi_n}{\partial x^n \partial x^k} 
\end{pmatrix} \begin{pmatrix}
    V^1 \\ \vdots \\ V^n
\end{pmatrix} = (\partial_{x^k} D\Phi)V.
\end{equation}
\end{lemma}

\begin{proof}
Given a coordinate system $(x^1, \ldots, x^n)$ of $M$, consider the coordinate tangent frame $(\partial_{x^1}, \ldots, \partial_{x^n})$ and coordinate coframe $(dx^1, \ldots, dx^n)$ of differential forms, and write $\Phi$ in coordinates as $\Phi = (\Phi_1, \ldots, \Phi_n)$. At each $\xi \in M$, the second derivative of $\Phi$ is a bilinear map $(T_\xi M) \times (T_\xi M) \to T_{\Phi(\xi)} M$. Expressing $D^2\Phi$ in terms of this coordinate system as in \eqref{eq:bilinear-matrix}, we have: 
\begin{equation}\label{eq:second-deriv-coords}
D^2\Phi = \begin{pmatrix} \displaystyle
    \sum_{i,j} \frac{\partial^2 \Phi_1}{\partial x^i \partial x^j} dx^i \otimes dx^j \\ \vdots \\ 
     \displaystyle\sum_{i,j} \frac{\partial^2 \Phi_n}{\partial x^i \partial x^j} dx^i \otimes dx^j
\end{pmatrix}.
\end{equation}
Consider now a vector field $V$ on $M$ (for example $V$ could be the time-$1$ stable direction $e^{(1)}$ for a set of time-$1$ hyperbolic coordinates $\{e^{(1)}, f^{(1)}\}$, as in Definition \ref{def:hypcoords}). In a unit coordinate system, write $V = V^1 \partial_{x^1} + \cdots + V^n \partial_{x^n}$, $V^k : M \to \mathbb R$ smooth functions. Putting $V$ into one of the arguments in \eqref{eq:second-deriv-coords} and expressing $D^2\Phi(V,\cdot)$ as in \eqref{eq:vector-into-bilinear}, we note: 
\begin{equation}\label{eq:vector-into-second-derivative}
D^2\Phi(V,\cdot) = \begin{pmatrix}
    \displaystyle \sum_{i=1}^n \frac{\partial^2 \Phi_1}{\partial x^i \partial x^1} V^i & \cdots & \displaystyle \sum_{i=1}^n \frac{\partial^2 \Phi_1}{\partial x^i \partial x^n} V^i \\
    \vdots & & \vdots \\
    \displaystyle \sum_{i=1}^n \frac{\partial^2 \Phi_n}{\partial x^i \partial x^1} V^i & \cdots & \displaystyle \sum_{i=1}^n \frac{\partial^2 \Phi_n}{\partial x^i \partial x^n} V^i
\end{pmatrix}.
\end{equation}
Putting a coordinate tangent vector $\partial_{x^k}$ into $D^2\Phi(V,\cdot)$ gives us: 
\begin{equation}\label{eq:column-second-deriv-coords-x}
D^2\Phi(V, \partial_{x^k}) = \begin{pmatrix} \displaystyle \sum_{i=1}^n \frac{\partial^2 \Phi_1}{\partial x^i \partial x^k} V^i \\ \vdots \\ \displaystyle \sum_{i=1}^n \frac{\partial^2 \Phi_n}{\partial x^i \partial x^k} V^i 
\end{pmatrix} = \begin{pmatrix}
    \displaystyle \frac{\partial^2 \Phi_1}{\partial x^1 \partial x^k} & \cdots & \displaystyle \frac{\partial^2 \Phi_1}{\partial x^n \partial x^k} \\
    \vdots & & \vdots \\
    \displaystyle \frac{\partial^2 \Phi_n}{\partial x^1 \partial x^k} & \cdots & \displaystyle \frac{\partial^2 \Phi_n}{\partial x^n \partial x^k} 
\end{pmatrix} \begin{pmatrix}
    V^1 \\ \vdots \\ V^n
\end{pmatrix} = (\partial_{x^k} D\Phi)V,
\end{equation}
where $\partial_{x^k} D\Phi = \partial_{x^k}(D\Phi_\xi) : T_\xi M \to T_{\Phi(\xi)}M$ (at each $\xi \in M$ where these coordinates are defined) is the linear transformation given by 
\begin{equation}\label{eq:partial-dphi-coords-x}
\partial_{x^k}(D\Phi_\xi) = \begin{pmatrix}
    \displaystyle \frac{\partial^2 \Phi_1}{\partial x^1 \partial x^k}(\xi) & \cdots & \displaystyle \frac{\partial^2 \Phi_1}{\partial x^n \partial x^k}(\xi) \\
    \vdots & & \vdots \\
    \displaystyle \frac{\partial^2 \Phi_n}{\partial x^1 \partial x^k}(\xi) & \cdots & \displaystyle \frac{\partial^2 \Phi_n}{\partial x^n \partial x^k} (\xi)
\end{pmatrix}.
\end{equation}
\end{proof}

In order to bound $\|D^2\Phi_{\xi_0}\|$ in terms of the derivatives $\partial_{x^k}(D\Phi_{\xi_0})$, we will compare $D^2\Phi_{\xi_0}$ to the monolinear map $D\Phi_{\xi_0}(\cdot, \partial_{x^k})$ (which is equal to $\partial_{x^k}(D\Phi_{\xi_0})$ by \eqref{eq:column-second-deriv-coords}). We first show how to estimate the norm of a general bilinear map $B : \mathbb R^n \times \mathbb R^n \to \mathbb R^n$ in terms of the norms of the maps $B(\mathbf v, \cdot)$.

\begin{lemma}\label{lem:bilinear}
    If $B : \mathbb R^n \times \mathbb R^n \to \mathbb R^n$ is a bilinear map, and $\mathcal B = (\mathbf u_1, \ldots,\mathbf u_n)$ is an orthonormal basis of $\mathbb R^n$, then for any $\mathbf v \in \mathbb R^n$:
    \begin{equation}\label{eq:bilinear-approx-1}
        \max_{1 \leq k \leq n} \|B(\mathbf v, \mathbf u_k)\| \leq \|B(\mathbf v, \cdot) \| \leq \sqrt n \max_{1 \leq k \leq n}\|B(\mathbf v, \mathbf u_k)\|
    \end{equation}
    Moreover: 
    \begin{equation}\label{eq:bilinear-approx-2}
        \max_{1 \leq k \leq n} \|B(\cdot, \mathbf u_k)\| \leq \|B\| \leq \sqrt n \max_{1 \leq k \leq n}\|B(\cdot, \mathbf u_k)\|
    \end{equation}
\end{lemma}

To prove Lemma \ref{lem:bilinear} we first prove a simple statement about linear maps. 

\begin{sublemma}\label{lem:matrix-norm-bound}
Let $A$ be an $m \times n$ matrix with real or complex entries, whose columns are $\mathbf a_1, \ldots, \mathbf a_n$ in terms of an orthonormal basis $\mathcal B$. Then we have: 
\[
\max_{1 \leq k\leq n} \|\mathbf a_k\| \leq \|A\| \leq \sqrt n \max_{1 \leq k \leq n} \|\mathbf a_k\|.
\]
\end{sublemma}

\begin{proof}
 Let $\mathcal B = \{\mathbf u_1, \ldots, \mathbf u_n\}$ be an orthonormal basis of $\mathbb R^n$. With respect to this basis, write $A = \begin{pmatrix}
    \mathbf a_1 & \cdots & \mathbf a_n
\end{pmatrix}$, with $\mathbf a_k$ the $k$th column of $A$. Then $A\mathbf u_k = \mathbf a_k$ for all $k$. In particular, $\|A\| \geq \|\mathbf a_k\|$ for all $k$, giving us the left hand side of the statement in the lemma. 
Meanwhile, writing $\mathbf v = v_1 \mathbf u_1 + \cdots + v_n \mathbf u_n$, we have: 
\[
\|A\| = \sup_{\|\mathbf v\| = 1} \|A\mathbf v\| \leq \sup_{\|\mathbf v\| =1 } \sum_{k=1}^n \|\mathbf a_k v_k\| \leq \left( \max_{1 \leq k \leq n} \|\mathbf a_k\|\right) \sup_{\|\mathbf v\| = 1} \left( |v_1| + \cdots + |v_n|\right).
\]
One can use Lagrange multipliers to show that the function $f(v_1, \ldots, v_n) = v_1 + \cdots + v_n$ restricted to $v_1^2 + \cdots + v_n^2 = 1$, $v_k \geq 1$, has maximal value $\sqrt n$. This gives us the second inequality of the lemma.
\end{proof}

\begin{proof}[Proof of Lemma \ref{lem:bilinear}]
    Applying Sublemma \ref{lem:matrix-norm-bound} to the linear map $B(\mathbf v, \cdot)$, we obtain \eqref{eq:bilinear-approx-1} immediately. Next, noting $\|B\| = \sup_{\|\mathbf x\| = \|\mathbf y \| = 1} \|B(\mathbf x, \mathbf y)\|$ for a bilinear map $B : \mathbb R^n \times \mathbb R^n \to \mathbb R^n$, we clearly have the first inequality in \eqref{eq:bilinear-approx-2}. For the second inequality of \eqref{eq:bilinear-approx-2}, we use \eqref{eq:bilinear-approx-1} to conclude: 
    \[
    \|B\| = \max_{\|\mathbf v\| = 1} \|B(\mathbf v,\cdot)\| \leq \sqrt n \max_{\|\mathbf v\| = 1} \max_{1 \leq k \leq n} \|B(\mathbf v, \mathbf u_k)\| \leq \sqrt n \max_{1 \leq k \leq n} \|B(\cdot, \mathbf u_k)\|.
    \]
\end{proof}

\begin{proof}[Proof of Proposition \ref{lem:V-into-second-deriv-approx}]
Both \eqref{eq:D2-Phi-v-approx} and \eqref{eq:D^2-Phi-approx} follow from \eqref{eq:bilinear-approx-1} and \eqref{eq:bilinear-approx-2} applied to $D^2\Phi_{\xi_0}(v,\cdot)$ and $D^2\Phi_{\xi_0}$, respectively, noting $(\partial_{x^1}(\xi_0), \ldots,\partial_{x^n}(\xi_0))$ forms an orthonormal basis of $T_{\xi_0}M$ and that $D^2\Phi_{\xi_0}(v,\partial_{x^k}) = (\partial_{x^k} D\Phi_{\xi_0})v$ by \eqref{eq:column-second-deriv-coords}.
\end{proof}

\section{Slow Variation: A Priori Bounds}
\label{sec-general} 

We now begin the proof of the first and main bound in the statement of Theorem \ref{thm:main-phase}. In this section we establish \emph{a priori} bounds on $\|Df^{(k)}\|$ which just rely on the existence of hyperbolic coordinates. In Section~\ref{sec-slow}  we apply the properties of quasi-hyperbolic points to the a priori estimates to obtain the first inequality in the statement of Theorem \ref{thm:main-phase}.

 We emphasize  that the first inequality in Theorem \ref{thm:main-phase} is independent of the choice of coordinates. However, for the proof it is  natural to use an appropriate coordinate system. 
Thus for  the remainder of this section we fix once and for all a Riemannian  coordinate system $(x,y)$  based at $\xi_0$, at which the coordinate tangent vectors $\{\partial_x(\xi_0), \partial_y(\xi_0)\}$ form an orthonormal basis at $T_{\xi_0}M$, though we emphasize that\emph{ the constants $K_{1}$ and $K_{2}$ in Theorem \ref{thm:main-phase} will not depend on this choice of coordinates}. 
Then, for $\varsigma = x,y$, we have 
   \begin{equation}\label{eq:partial-Dphii}
    \partial_\varsigma \left(D\Phi_{\xi_0}\right)= \left( \begin{array}{cc}
\partial_{\varsigma x} \Phi_1(\xi_0) & \partial_{\varsigma y} \Phi_1(\xi_0) \\
\partial_{\varsigma x} \Phi_2(\xi_0) & \partial_{\varsigma y} \Phi_2(\xi_0) 
    \end{array}\right). 
    \end{equation}
The entries of the matrix \eqref{eq:partial-Dphii} are  just the second order partial derivatives of \( \Phi\) with respect to this coordinate system.   For simplicity, we use the following notation:
   \[
   A_{k}:=\frac{\sqrt 2\|f^{(k)}_k\|^2}{\|f^{(k)}_k\|^2 - \|e^{(k)}_k\|^2}\quand B_{k}:=\frac{\sqrt 2\|e^{(k)}_k\|^2}{\|f^{(k)}_k\|^2 - \|e^{(k)}_k\|^2}.
   \]
\[
\mathfrak E^{(k)}_{i}:=\frac{\|[\partial_{\varsigma}(D\Phi_{\xi_{i}})]e^{(k)}_{i}\|\|e^{(k)}_{i+1}\|}{|\det(D\Phi^{i+1}_{\xi_0})|}\quand \mathfrak F^{(k)}_{i}:=\frac{\|[\partial_{\varsigma}(D\Phi_{\xi_{i}})]f^{(k)}_{i}\|\|f^{(k)}_{i+1}\|}{|\det(D\Phi^{i+1}_{\xi_0})|}
\]
The following proposition  requires no assumptions except that the map \(  \Phi^{k}  \) is \(  C^{2}  \) at the  point \(  \xi_{0}  \) and that hyperbolic coordinates  \(  e^{(k)}, f^{(k)}  \) are defined, which we now assume for the rest of this section.

\begin{proposition}\label{prop:apriori}
For \(\varsigma=x,y\) we have
\begin{equation}\label{eq-apriori}
    \|Df^{(k)}\| \leq A_{k} 
\mathfrak E^{(k)}_{0}+A_{k}\sum_{i=1}^{k-1}\mathfrak E^{(k)}_{i}+
B_{k} \sum_{i=0}^{k-1}\mathfrak F^{(k)}_{i}.
\end{equation}

% \todelete{\begin{equation}\label{eq-apriori-x}
% \|Df^{(k)}\|
% \leq 
% \sum_{i=0}^{k-1}
% \frac{\|\partial_{\varsigma}[D\Phi_{\xi_{i}}]e^{(k)}_{i}\|\cdot\|e^{(k)}_{i+1}\|}{|\det(D\Phi^{i+1}_{\xi_0})|}
% +
% \frac{\|e^{(k)}_k\|^2}{\|f^{(k)}_k\|^2}\sum_{i=0}^{k-1}\frac{\|\partial_{\varsigma}[D\Phi_{\xi_{i}}]f^{(k)}_{i}\|\cdot\|f^{(k)}_{i+1}\|}{|\det(D\Phi^{i+1}_{\xi_0})|}.
% \end{equation}}
\end{proposition}

The remainder of this section is devoted to the proof of Proposition \ref{prop:apriori}. First of all,  in view of \eqref{eq:hypcorest} and Remark \ref{rmk:SVD}, we define the linear map $\mathcal L^{(k)}_{\xi_0} : T_{\xi_0} M \to T_{\xi_0} M$ by:
\begin{equation}\label{eq:L-def}
    \mathcal L^{(k)}_{\xi_0} = \mathcal L^{(k)} = (D\Phi^k_{\xi_0})^* \circ D\Phi_{\xi_0}^k
\end{equation} 
for which we have 
\begin{equation}\label{eq:L-eigens}
    \mathcal L^{(k)} e^{(k)} = \|e^{(k)}_k\|^2 e^{(k)} \quad \textrm{and} \quad \mathcal L^{(k)} f^{(k)} = \|f^{(k)}_k\|^2 f^{(k)}.
\end{equation}
We omit explicit reference to $\xi_0$ in the interest of clarity. 
We will also write the components of the unit vector field $f^{(k)}$ as $f^{(k)} = (f^{(k)}_1, f^{(k)}_2)$, and then write the covariant derivative of $f^{(k)}$ in $\mathbb R^2$ as 
\[
Df^{(k)} = \left(\partial_x f^{(k)}, \partial_y f^{(k)}\right),
\]
 where \begin{equation}\label{eq:Df-columns}
\partial_x f^{(k)} = \left( \partial_x f^{(k)}_1, \partial_x f^{(k)}_2\right) \quad \textrm{and} \quad \partial_y f^{(k)} = \left( \partial_y f^{(k)}_1, \partial_y f^{(k)}_2\right)
\end{equation}
are the columns of $Df^{(k)}$. We split the proof into three Lemmas.

\begin{lemma}\label{lem:Df-apriori-max} 
\begin{equation}\label{eq:Dfbound}
\|Df^{(k)}\| \leq \sqrt 2 \max_{\varsigma = x,y} \left\{ | \langle e^{(k)}, \partial_{\varsigma} f^{(k)} \rangle |\right\}.
\end{equation}
\end{lemma}

\begin{proof} Notice that by Sublemma \ref{lem:matrix-norm-bound}, putting $A = Df^{(k)}$, and $\mathbf a_1 = \partial_x f^{(k)}$ and $\mathbf a_2 = \partial_y f^{(k)}$, we have 
\begin{equation}\label{eq:Dfk-bd-1}
\|Df^{(k)}\| \leq \sqrt 2 \max_{\varsigma = x,y} \left\{ \partial_{\varsigma} f^{(k)}\right\}.
\end{equation}
Additionally, by orthonormality of $\{e^{(k)}, f^{(k)}\}$, we obtain: 
\begin{equation}\label{eq:linear-algebra}
\partial_\varsigma f^{(k)} = \langle e^{(k)}, \partial_\varsigma f^{(k)} \rangle e^{(k)} + \langle f^{(k)}, \partial_\varsigma f^{(k)} \rangle f^{(k)}.
\end{equation}
Writing the columns of $Df^{(k)}$ as in \eqref{eq:Df-columns}, equations \eqref{eq:Dfk-bd-1} and \eqref{eq:linear-algebra} together give us:
\[
\|Df^{(k)}\|\leq \sqrt 2 \max_{\varsigma = x,y}\{\|\partial_{\varsigma}f^{(k)}\| \} 
\leq \sqrt 2 \max_{\varsigma = x,y}\{   | \langle e^{(k)},\partial_{\varsigma}f^{(k)}\rangle| + 
| \langle f^{(k)},\partial_{\varsigma}f^{(k)}\rangle| 
\}.
\]
Differentiating the equality \(\|f^{(k)}\|^2=\langle f^{(k)}, f^{(k)}\rangle=1\) we get that \( | \langle f^{(k)},\partial_{\varsigma}f^{(k)}\rangle| = 0 \), thus obtaining~\eqref{eq:Dfbound}.
\end{proof}

\begin{lemma}\label{lem:Dfboundnew}
For $\varsigma = x,y$,
\begin{equation}\label{eq:Dfboundnew}
|\langle e^{(k)}, \partial_\varsigma f^{(k)}\rangle|
\leq 
\frac{|\langle e^{(k)}, ( \partial_\varsigma \mathcal L^{(k)}) f^{(k)}\rangle|}{\|f^{(k)}_k\|^2 - \|e^{(k)}_k\|^2}.
\end{equation}
% \todelete{
% \begin{equation}\label{eq:Dfboundnew-x}
% \|Df^{(k)}\|
% \leq 
% \max_{\varsigma = x,y}\left\{\frac{2\langle e^{(k)}, ( \partial_\varsigma \mathcal L^{(k)}) f^{(k)}\rangle}{\|f^{(k)}_k\|^2}\right\}.
% \end{equation}
% }
\end{lemma}

\begin{proof}
By differentiating the second equation in \eqref{eq:L-eigens}, we have
%\[   \partial_{\varsigma}[(D\Phi^k_{\xi_0})^*\circ(D\Phi^k_{\xi_0})]f^{(k)}(\xi_{0}) +(D\Phi^k_{\xi_0})^*\circ(D\Phi^k_{\xi_0})\partial_{\varsigma}f^{(k)}(\xi_0)=\partial_{\varsigma}[\|f^{(k)}_k(\xi_k)\|^2]f^{(k)}(\xi_0)+\|f^{(k)}_k(\xi_k)\|^2\partial_{\varsigma}f^{(k)}(\xi_0).
%\]
\[   (\partial_{\varsigma}\mathcal L^{(k)})f^{(k)} +\mathcal L^{(k)}\partial_{\varsigma}f^{(k)}=(\partial_{\varsigma}\|f^{(k)}_k\|^2)f^{(k)}+\|f^{(k)}_k\|^2\partial_{\varsigma}f^{(k)}.
\]
Taking the scalar product with $e^{(k)}$, and using the fact that \(  e^{(k)}, f^{(k)}  \) are orthogonal, gives
\begin{equation}\label{eq-eigen2}
 \langle e^{(k)},(\partial_{\varsigma}\mathcal L^{(k)})f^{(k)}\rangle+\langle e^{(k)},\mathcal L^{(k)}\partial_{\varsigma}f^{(k)}\rangle=\langle e^{(k)},\|f^{(k)}_k\|^2\partial_{\varsigma}f^{(k)}\rangle.
\end{equation}
Using the first equation in \eqref{eq:L-eigens} and the fact that the matrix \(\mathcal L^{(k)} = (D\Phi^k_{\xi_0})^*\circ(D\Phi^k_{\xi_0})\) is self-adjoint, we have
$$\langle e^{(k)},\mathcal L^{(k)}\partial_{\varsigma}f^{(k)}\rangle=\langle\mathcal L^{(k)}e^{(k)},\partial_{\varsigma}f^{(k)}\rangle=\|e^{(k)}_k\|^2\langle e^{(k)}, \partial_{\varsigma}f^{(k)}\rangle.$$ 
Substituting this last equality into \eqref{eq-eigen2} gives
\[
 \langle e^{(k)},(\partial_{\varsigma}\mathcal L^{(k)})f^{(k)}\rangle+\|e^{(k)}_k\|^2\langle e^{(k)},\partial_{\varsigma}f^{(k)}\rangle=\|f^{(k)}_k\|^2\langle e^{(k)},\partial_{\varsigma}f^{(k)}\rangle,
\]
which gives \eqref{eq:Dfboundnew}.
% \todelete{By \eqref{eq:fk-ek}, it follows that
% \[
% \langle e^{(k)}, \partial_\varsigma f^{(k)}\rangle \leq \frac{4\langle e^{(k)}, (\partial_\varsigma \mathcal L^{(k)}) f^{(k)}\rangle}{3\|f^{(k)}_k\|^2}.
% \]}
%\todelete{The lemma now follows from \eqref{eq:Dfbound}.}
\end{proof}

\begin{lemma}\label{lem:eq-1}
For \(\varsigma=x,y\),
\begin{equation}\label{eq-1}
    |\langle e^{(k)}, (\partial_\varsigma \mathcal L^{(k)})f^{(k)}\rangle| \leq \|f^{(k)}_k\|^2\sum_{i=0}^{k-1}
\frac{\|[\partial_{\varsigma}(D\Phi_{\xi_{i}})]e^{(k)}_{i}\|\|e^{(k)}_{i+1}\|}{|\det(D\Phi^{i+1}_{\xi_0})|}
+
\|e^{(k)}_k\|^2\sum_{i=0}^{k-1}\frac{\|[\partial_{\varsigma}(D\Phi_{\xi_{i}})]f^{(k)}_{i}\|\|f^{(k)}_{i+1}\|}{|\det(D\Phi^{i+1}_{\xi_0})|}.
\end{equation}

% \todelete{\begin{equation}\label{eq-1-x}
% \begin{aligned}
% \frac{|\langle e^{(k)}, \mathcal L_\varsigma f^{(k)}\rangle|}{\|f^{(k)}_k\|^2}
% \leq
% \sum_{i=0}^{k-1}
% \frac{\|\partial_{\varsigma}[D\Phi_{\xi_{i}}]e^{(k)}_{i}\|\cdot\|e^{(k)}_{i+1}\|}{|det(D\Phi^{i+1}_{\xi_0})|}
% +
% \frac{\|e^{(k)}_k\|^2}{\|f^{(k)}_k\|^2}\sum_{i=0}^{k-1}\frac{\|\partial_{\varsigma}[D\Phi_{\xi_{i}}]f^{(k)}_{i}\|\cdot\|f^{(k)}_{i+1}\|}{|det(D\Phi^{i+1}_{\xi_0})|}.
% \end{aligned}
% \end{equation}}
\end{lemma}

\begin{proof}
By the Leibniz product rule, we have: 
\begin{equation}\label{eq:leibniz-1}
\partial_{\varsigma}\mathcal L^{(k)}=  \partial_{\varsigma}[(D\Phi^k_{\xi_0})^*(D\Phi^k_{\xi_0})]=[\partial_{\varsigma}(D\Phi^k_{\xi_0})]^*(D\Phi^k_{\xi_0})+(D\Phi^k_{\xi_0})^* [\partial_{\varsigma}(D\Phi^k_{\xi_0})].
\end{equation}
By inductively applying the Leibniz rule to $\partial_\varsigma(D\Phi_{\xi_0}^k) = \partial_\varsigma( D\Phi_{\xi_{k-1}} \cdots D\Phi_{\xi_0})$, we obtain:
\begin{equation}\label{eq:leibniz-k}
\partial_\varsigma(D\Phi_{\xi_0}^k) = \sum_{i=0}^{k-1} (D\Phi_{\xi_{i+1}}^{k-i-1}) [\partial_\varsigma (D\Phi_{\xi_i})] (D\Phi_{\xi_0}^i).
\end{equation}
By \eqref{eq:leibniz-1} and \eqref{eq:leibniz-k}:
\[
\begin{aligned}
\partial_{\varsigma}\mathcal L^{(k)}
&=\left(\sum_{i=0}^{k-1}(D\Phi^{k-i-1}_{\xi_{i+1}})[\partial_{\varsigma}(D\Phi_{\xi_{i}})](D\Phi^{i}_{\xi_0})\right)^*(D\Phi^k_{\xi_0})+(D\Phi^k_{\xi_0})^*\left(\sum_{i=0}^{k-1}(D\Phi^{k-i-1}_{\xi_{i+1}})[\partial_{\varsigma}(D\Phi_{\xi_{i}})](D\Phi^{i}_{\xi_0})\right).
\end{aligned}
\]
Applying $\partial_\varsigma \mathcal L^{(k)}$ to $f^{(k)}$ and taking the scalar product with $e^{(k)}$ gives:
\begin{equation}\label{eq:bigsum-1}
 \begin{aligned}
 \langle e^{(k)},\:(\partial_{\varsigma}\mathcal L^{(k)}) f^{(k)}\rangle
 =&\bigg\langle e^{(k)},\:\left(\sum_{i=0}^{k-1}(D\Phi^{k-i-1}_{\xi_{i+1}})[\partial_{\varsigma}(D\Phi_{\xi_{i}})](D\Phi^{i}_{\xi_0})\right)^*(D\Phi^k_{\xi_0})f^{(k)}\bigg\rangle\\
&+\bigg\langle e^{(k)},\: (D\Phi^k_{\xi_0})^*\left(\sum_{i=0}^{k-1}(D\Phi^{k-i-1}_{\xi_{i+1}})[\partial_{\varsigma}(D\Phi_{\xi_{i}})](D\Phi^{i}_{\xi_0})\right)f^{k)}\bigg\rangle\\
=&\sum_{i=0}^{k-1}\bigg\langle e^{(k)},\:\left((D\Phi^{k-i-1}_{\xi_{i+1}})[\partial_{\varsigma}(D\Phi_{\xi_{i}})](D\Phi^{i}_{\xi_0})\right)^*(D\Phi^k_{\xi_0})f^{(k)}\bigg\rangle\\
&+\sum_{i=0}^{k-1}\bigg\langle e^{(k)}, \:(D\Phi^k_{\xi_0})^*(D\Phi^{k-i-1}_{\xi_{i+1}})[\partial_{\varsigma}(D\Phi_{\xi_{i}})](D\Phi^{i}_{\xi_0})f^{k)}\bigg\rangle \\
&= \sum_{i=0}^{k-1} \mathfrak{A}_i + \sum_{i=0}^{k-1} \mathfrak{B}_i,%\\
%=&\todelete{\sum_{i=0}^{k-1}\Big\langle (D\Phi^{k-i-1}_{\xi_{i+1}}) [\partial_{\varsigma}(D\Phi_{\xi_{i}})] (D\Phi^{i}_{\xi_0})e^{(k)},\:D\Phi^k_{\xi_0}f^{(k)}\Big\rangle}\\
%&\todelete{+\sum_{i=0}^{k-1}\Big\langle D\Phi^k_{\xi_0}e^{(k)},\:(D\Phi^{k-i-1}_{\xi_{i+1}}) [\partial_{\varsigma}(D\Phi_{\xi_{i}})] (D\Phi^{i}_{\xi_0})f^{(k)}\Big\rangle,}
\end{aligned}
\end{equation}
where:
\begin{align*} 
\mathfrak A_i &= \bigg\langle e^{(k)},\:\left((D\Phi^{k-i-1}_{\xi_{i+1}})[\partial_{\varsigma}(D\Phi_{\xi_{i}})](D\Phi^{i}_{\xi_0})\right)^*(D\Phi^k_{\xi_0})f^{(k)}\bigg\rangle,\\
\mathfrak B_i &= \bigg\langle e^{(k)}, \:(D\Phi^k_{\xi_0})^*(D\Phi^{k-i-1}_{\xi_{i+1}})[\partial_{\varsigma}(D\Phi_{\xi_{i}})](D\Phi^{i}_{\xi_0})f^{k)}\bigg\rangle.
\end{align*}
Consider the summands in the first sum of the right hand side of \eqref{eq:bigsum-1}. We have:
\begin{equation}\label{eq:cleanup-f1}
\begin{aligned}
\mathfrak A_i
&= \Big\langle (D\Phi^{k-i-1}_{\xi_{i+1}}) [\partial_{\varsigma}(D\Phi_{\xi_{i}})] (D\Phi^{i}_{\xi_0})e^{(k)},\:D\Phi^k_{\xi_0}f^{(k)}\Big\rangle \\
&= \Big\langle[\partial_{\varsigma}D\Phi_{\xi_{i}}]( D\Phi^{i}_{\xi_0})e^{(k)},\:(D\Phi^{k-i-1}_{\xi_{i+1}})^* (D\Phi^k_{\xi_0})f^{(k)}\Big\rangle \\
&= \Big\langle[\partial_{\varsigma}D\Phi_{\xi_{i}}]e^{(k)}_i,\:(D\Phi^{k-i-1}_{\xi_{i+1}})^* (D\Phi^k_{\xi_0})f^{(k)}\Big\rangle.
\end{aligned}
\end{equation}
%\todelete{which can be rewritten as
%\begin{equation}\label{eq-first}
%\begin{aligned}
% \langle e^{(k)},(\partial_\varsigma \mathcal L^{(k)}) f^{(k)}\rangle&=\sum_{i=0}^{k-1}\langle[\partial_{\varsigma}D\Phi_{\xi_{i}}]( D\Phi^{i}_{\xi_0})e^{(k)},(D\Phi^{k-i-1}_{\xi_{i+1}})^* (D\Phi^k_{\xi_0})f^{(k)})\rangle\\
%&+\sum_{i=0}^{k-1}\langle(D\Phi^{k-i-1}_{\xi_{i+1}})^* (D\Phi^k_{\xi_0})e^{(k)},[\partial_{\varsigma}(D\Phi_{\xi_{i}})] (D\Phi^{i}_{\xi_0})f^{(k)}\rangle.
%\end{aligned}
%\end{equation}}
Observing that $D\Phi_{\xi_{i+1}}^{k-i-1} = D\Phi_{\xi_0}^{k} (D\Phi_{\xi_{0}}^{i+1})^{-1}$,
we obtain:
\begin{equation}\label{eq-ob}
 (D\Phi^{k-i-1}_{\xi_{i+1}})^*=\left[(D\Phi^{i+1}_{\xi_0})^{-1}\right]^*(D\Phi^k_{\xi_0})^*.
\end{equation}
Therefore, recalling that $\mathcal L^{(k)} f^{(k)} = (D\Phi_{\xi_0}^k)^*(D\Phi_{\xi_0}^k) f^{(k)} = \|f^{(k)}_k\|^2 f^{(k)}$, we obtain the expression:
\begin{align*}
(D\Phi^{k-i-1}_{\xi_{i+1}})^* (D\Phi^k_{\xi_0})f^{(k)} &= \big[(D\Phi^{i+1}_{\xi_0})^{-1}\big]^*(D\Phi^k_{\xi_0})^*(D\Phi^k_{\xi_0})f^{(k)} = \|f^{(k)}_{k}\|^{2}
\big[(D\Phi^{i+1}_{\xi_0})^{-1}\big]^* f^{(k)}.
\end{align*}
Applying this to the second argument of the right hand side of \eqref{eq:cleanup-f1}, we get: 
\begin{equation}\label{eq:cleanup-f2}
\begin{aligned}
\mathfrak A_i
&= \|f^{(k)}_k\|^2\Big\langle[\partial_{\varsigma}D\Phi_{\xi_{i}}]e^{(k)}_i,\:\big[(D\Phi^{i+1}_{\xi_0})^{-1}\big]^* f^{(k)}\Big\rangle = \|f^{(k)}_k\|^2\Big\langle(D\Phi^{i+1}_{\xi_0})^{-1}[\partial_{\varsigma}D\Phi_{\xi_{i}}]e^{(k)}_i,\:f^{(k)}\Big\rangle.
\end{aligned}
\end{equation}
To estimate this inner product in \eqref{eq:cleanup-f2}, we use two basic properties of the determinant: given a $2 \times 2$ matrix $A$ and two vectors $v$, $w$, we have:
\begin{equation}\label{eq:det-scalar-product}
    \det(Av, Aw) = \det(A)\det(v,w) \quad \textrm{and} \quad \det(v,w) = \pm \langle v, w^\top\rangle 
\end{equation}
where $\det(v,w)$ refers to the determinant of a $2 \times 2$ matrix whose columns are the vectors $v,w$, and $w^\top$ is a vector orthogonal to $w$ with $\|w\| = \|w^\top\|$. Taking $(f^{(k)})^\top = e^{(k)}$, applying these properties to \eqref{eq:cleanup-f2} gives us: 
\begin{equation}\label{eq:cleanup-f3}
\begin{aligned}
\mathfrak A_i
&= \|f^{(k)}_k\|^2\Big\langle(D\Phi^{i+1}_{\xi_0})^{-1}[\partial_{\varsigma}D\Phi_{\xi_{i}}]e^{(k)}_i,\:f^{(k)}\Big\rangle \\
&= \pm \|f^{(k)}_k\|^2\det \Big((D\Phi^{i+1}_{\xi_0})^{-1}[\partial_{\varsigma}D\Phi_{\xi_{i}}]e^{(k)}_i,\:(f^{(k)})^\top\Big)\\
&= \pm \|f^{(k)}_k\|^2\det \Big((D\Phi^{i+1}_{\xi_0})^{-1}[\partial_{\varsigma}D\Phi_{\xi_{i}}]e^{(k)}_i,\:e^{(k)}\Big) \\
&= \pm \|f^{(k)}_k\|^2\det \Big((D\Phi^{i+1}_{\xi_0})^{-1}[\partial_{\varsigma}D\Phi_{\xi_{i}}]e^{(k)}_{i},\:(D\Phi^{i+1}_{\xi_0})^{-1} e^{(k)}_{i+1}\Big) \\
&= \pm \|f^{(k)}_k\|^2 \det (D\Phi_{\xi_0}^{i+1})^{-1} \det \Big([\partial_{\varsigma}D\Phi_{\xi_{i}}]e^{(k)}_{i},\: e^{(k)}_{i+1}\Big).
\end{aligned}
\end{equation}
By \eqref{eq:det-scalar-product} and the Cauchy-Schwarz inequality, we have 
\begin{equation}\label{eq:det-vw-est}
    |\det(v,w)| = |\langle v, w^\top\rangle| \leq \|v\|\|w\|.
\end{equation}
Taking the absolute value of both sides of \eqref{eq:cleanup-f3} and applying \eqref{eq:det-vw-est} to $\det \big([\partial_{\varsigma}D\Phi_{\xi_{i}}]e^{(k)}_{i},\: e^{(k)}_{i+1}\big)$, we obtain: 
\begin{equation}\label{eq:cleanup-f4}
\begin{aligned}
|\mathfrak A_i|\leq \|f^{(k)}_k\|^2 \frac{\|\partial_\varsigma [D\Phi_{\xi_i}] e_i^{(k)}\| \|e^{(k)}_{i+1}\|}{|\det(D\Phi_{\xi_0}^{i+1})|}.
\end{aligned}
\end{equation}
These are the summands in the first sum in \eqref{eq-1}.
We perform a similar calculation on the summands of the second sum on the right hand side of \eqref{eq:bigsum-1}. We have: 
\begin{equation}\label{eq:cleanup-e1}
\begin{aligned}
    \mathfrak B_i
    &= \Big\langle D\Phi^k_{\xi_0}e^{(k)},\:(D\Phi^{k-i-1}_{\xi_{i+1}}) [\partial_{\varsigma}(D\Phi_{\xi_{i}})] (D\Phi^{i}_{\xi_0})f^{(k)}\Big\rangle\\
    &= \Big\langle(D\Phi^{k-i-1}_{\xi_{i+1}})^* (D\Phi^k_{\xi_0})e^{(k)},\:[\partial_{\varsigma}(D\Phi_{\xi_{i}})] (D\Phi^{i}_{\xi_0})f^{(k)}\Big\rangle \\
    &= \Big\langle(D\Phi^{k-i-1}_{\xi_{i+1}})^* (D\Phi^k_{\xi_0})e^{(k)},\:[\partial_{\varsigma}(D\Phi_{\xi_{i}})] f^{(k)}_i\Big\rangle.
\end{aligned}
\end{equation}
Recall now that $e^{(k)}$ is an eigenvector of $(D\Phi_{\xi_0}^k)^*(D\Phi_{\xi_0}^k)$ with eigenvalue $\|e^{(k)}_k\|^2$. By \eqref{eq-ob}, we get: 
\[
(D\Phi_{\xi_{i+1}}^{k-i-1})^*(D\Phi_{\xi_0}^k)e^{(k)} = \big[ (D\Phi_{\xi_0}^{i+1})^{-1}\big]^*(D\Phi_{\xi_0}^k)^*(D\Phi_{\xi_0}^k)e^{(k)} = \|e^{(k)}_k\|^2 \big[ (D\Phi_{\xi_0}^{i+1})^{-1}\big]^* e^{(k)}.
\]
Applying this to the first argument of the right hand side of \eqref{eq:cleanup-e1}, we obtain: 
\begin{equation}\label{eq:cleanup-e2}
\begin{aligned}
    \mathfrak B_i
    &= \|e^{(k)}_k\|^2 \Big\langle \big[ (D\Phi_{\xi_0}^{i+1})^{-1}\big]^* e^{(k)},\:[\partial_{\varsigma}(D\Phi_{\xi_{i}})] f^{(k)}_i\Big\rangle \\
    &= \|e^{(k)}_k\|^2 \Big\langle e^{(k)},\:(D\Phi_{\xi_0}^{i+1})^{-1}[\partial_{\varsigma}(D\Phi_{\xi_{i}})] f^{(k)}_i\Big\rangle .
    \end{aligned}
\end{equation}
We again use \eqref{eq:det-scalar-product} to estimate this inner product, this time taking $(e^{(k)})^\top = f^{(k)}$: 
\begin{equation}\label{eq:cleanup-e3}
\begin{aligned}
\mathfrak B_i
&= \pm \|e^{(k)}_k\|^2 \det \Big( f^{(k)},\: (D\Phi_{\xi_0}^{i+1})^{-1}[\partial_\varsigma (D\Phi_{\xi_i})] f^{(k)}_i\Big) \\
&= \pm \|e^{(k)}_k\|^2 \det \Big( (D\Phi_{\xi_0}^{i+1})^{-1}f^{(k)}_{i+1},\: (D\Phi_{\xi_0}^{i+1})^{-1}[\partial_\varsigma (D\Phi_{\xi_i})] f^{(k)}_i\Big) \\
&= \pm \|e^{(k)}_k\|^2 \det(D\Phi_{\xi_0}^{i+1})^{-1} \det \Big(f^{(k)}_{i+1},\: [\partial_\varsigma (D\Phi_{\xi_i})] f^{(k)}_i\Big).
\end{aligned}
\end{equation}
Using \eqref{eq:det-vw-est}, taking the absolute value of both sides of \eqref{eq:cleanup-e3} gives us: 
\begin{equation}\label{eq:cleanup-e4}
\begin{aligned} 
|\mathfrak B_i| \leq \|e^{(k)}_k\|^2 \frac{\|[\partial_\varsigma(D\Phi_{\xi_i})]f^{(k)}_i\|\|f^{(k)}_{i+1}\|}{|\det(D\Phi_{\xi_{0}}^{i+1}|}.
\end{aligned}
\end{equation}
These are the summands of the second sum in \eqref{eq-1}. Now, \eqref{eq-1} follows after taking the absolute value of \eqref{eq:bigsum-1} and applying \eqref{eq:cleanup-f4} and \eqref{eq:cleanup-e4}. 
\end{proof}

\begin{proof}[Proof of Proposition \ref{prop:apriori}] The proposition is an immediate consequence of Lemmas \ref{lem:Df-apriori-max} - \ref{lem:eq-1}.
\end{proof}

\section{Slow Variation for Quasi-Hyperbolic Points}\label{sec-slow} 

We now consider the general bound \eqref{eq-apriori} obtained in Proposition \ref{prop:apriori}  and will use  the quasi-hyperbolicity conditions to get more explicit bounds for the three terms on the right hand side of  \eqref{eq-apriori}.  First of all, in  Section \ref{sec:aposteriori} we will prove the following.
\begin{proposition}\label{prop:aposteriori}
For \(\varsigma=x,y\) we have 
\begin{equation}\label{eq:aposteriori}
\|Df^{(k)}\| \leq K_{1} \left( \mathfrak E_0^{(k)} + \sum_{i=1}^{k-1} \mathfrak E_i^{(k)} + \frac{\|e^{(k)}_k\|^2}{\|f^{(k)}_k\|^2} \sum_{i=0}^{k-1} \mathfrak F_{i}^{(k)}\right). 
\end{equation}
\end{proposition}
In Section \ref{subsubsec:var-pf-I} and \ref{subsubsec:var-pf-II} we estimate the first two terms  on the right hand side of \eqref{eq:aposteriori} assuming quasi-hyperbolicity Conditions \eqref{eq:condition-I} and \eqref{eq:condition-II} respectively. Then in Section \ref{sec:Fik} we estimate the third term and finally, in Section \ref{sec:1stineq}, we combine these estimates to complete the proof of the first inequality in Theorem~\ref{thm:main-phase}.

\subsection{Proof of Proposition \ref{prop:aposteriori}}\label{sec:aposteriori}
Proposition \ref{prop:aposteriori} follows immediately from Proposition \ref{prop:apriori} and the following Lemma. 
\begin{lemma}
For every \( k \geq 1 \) we have 
\[
A_{k}\leq K_{1} \quand 
B_{k} \leq \frac{\|e^{(k)}_k\|^2}{\|f^{(k)}_k\|^2} K_{1}.
\]
\end{lemma}
\begin{proof}
By the definition of eccentricity and \eqref{eq:constants-basic-sing1} we have 
\( 
{\|e^{(k)}_k\|}/{\|f^{(k)}_k\|} = C_{\xi_0,k} \leq Bc^k, 
\) 
and therefore
\[
\|f^{(k)}_k\|^2 - \|e^{(k)}_k\|^2 \geq (1-B^{2}c^{2k})\|f^{(k)}_k\|^2 \geq (1-B^{2}c^2)\|f^{(k)}_k\|^2
\]
and so
\[
\frac{\|f^{(k)}_k\|^2}{\|f^{(k)}_k\|^2 - \|e^{(k)}_k\|^2} \leq \frac 1{1-B^{2}c^2} \quad \textrm{and} \quad \frac{\|e^{(k)}_k\|^2}{\|f^{(k)}_k\|^2 - \|e^{(k)}_k\|^2} \leq \frac{\|e^{(k)}_k\|^2}{(1-B^{2}c^2)\|f^{(k)}_k\|^2}.
\]
Recalling the definitions of $A_k$ and $B_k$, this implies 
\[
   A_{k}:=\frac{\sqrt 2\|f^{(k)}_k\|^2}{\|f^{(k)}_k\|^2 - \|e^{(k)}_k\|^2} \leq \frac{\sqrt 2}{1-B^{2}c^2} \quand B_{k}:=\frac{\sqrt 2\|e^{(k)}_k\|^2}{\|f^{(k)}_k\|^2 - \|e^{(k)}_k\|^2} \leq \frac{\sqrt 2\|e^{(k)}_k\|^2}{(1-B^{2}c^2)\|f^{(k)}_k\|^2}
\]
which gives the statement in the Lemma. 
\end{proof}

\subsection{Estimates for $\mathfrak E_i^{(k)}$ with Quasi-Hyperbolicity, Condition (I)}\label{subsubsec:var-pf-I}

We assume throughout this subsection that $\xi_0$ is singular quasi-hyperbolic up to time $k$ and satisfies \eqref{eq:condition-I}. 

\begin{lemma}\label{lem:var-pf-1I}
    \begin{equation}\label{eq:var-pf-1I}
        \mathfrak E_0^{(k)} \leq \|D^2 \Phi_{\xi_0}(e^{(1)},\cdot)\| + Q_{3}c
    \end{equation}
\end{lemma}

\begin{proof}
    By \eqref{eq:second-derivative-norm} and \eqref{eq:ex-platypus}, as well as by \eqref{eq:constants-basic2-sing1} and \eqref{eq:cvgce-k-quasi-hyp-1}:
    \begin{align*}
    \|[\partial_\varsigma (D\Phi_{\xi_0})]e^{(k)}\| &\leq \|[\partial_\varsigma (D\Phi_{\xi_0})] e^{(1)}\| + \|\partial_\varsigma(D\Phi_{\xi_0})\|\|e^{(k)} - e^{(1)}\| \\
    &\leq \|D^2 \Phi_{\xi_0}(e^{(1)},\cdot)\| + \|D^2 \Phi_{\xi_0}\| \|e^{(k)} - e^{(1)}\| \\
    &\leq \|D^2 \Phi_{\xi_0}(e^{(1)},\cdot)\| + \frac{Q_1D\Gamma^2 \widetilde \Gamma c}{\lambda}
    \end{align*}
\end{proof}

\begin{lemma}\label{lem:var-pf-2I}
    \begin{equation}\label{eq:var-pf-2I}
        \sum_{i=1}^{k-1} \mathfrak E_i^{(k)} \leq Q_{4}c.
    \end{equation}
\end{lemma}

\begin{proof}
    We observe that $\|[\partial_\varsigma (D\Phi_{\xi_i}) e^{(k)}_i\| \leq \|\partial_\varsigma (D\Phi_{\xi_i})\|\|e^{(k)}_i\|$. By \eqref{eq:constants-basic2-sing1} and \eqref{eq:second-derivative-norm}, we get: 
    \begin{equation}\label{eq:partial-DPhi-typeI}
\|\partial_{\varsigma}(D\Phi_{\xi_{i}})\| \leq \|D^2\Phi_{\xi_i}\| \leq D\Gamma\widetilde\Gamma^i.
\end{equation}
From \eqref{eq:cvgce-k-quasi-hyp-2} and \eqref{eq:partial-DPhi-typeI}, we get:
\begin{equation}\label{eq:dphi-deriv-bounds-typeI}
\|[\partial_\varsigma (D\Phi_{\xi_i})] e_i^{(k)}\| \leq Q_1 \left( \frac{\Gamma^2 \widetilde \Gamma c}{\lambda}\right)^i D\Gamma \widetilde \Gamma^i = Q_1D\Gamma \left( \frac{\Gamma^2 \widetilde \Gamma^2 c}{\lambda}\right)^i.
\end{equation}
By \eqref{eq:cvgce-k-quasi-hyp-3} and \eqref{eq:dphi-deriv-bounds-typeI}, we get:
\begin{equation}\label{eq:i geq 1 typeI}
\mathfrak E_i^{(k)} = \frac{\|[\partial_\varsigma(D\Phi_{\xi_i})]e^{(k)}_i\|\|e^{(k)}_{i+1}\|}{|\det(D\Phi_{\xi_0}^{i+1})|} \leq Q_1 D\Gamma \left(\frac{\Gamma^2\widetilde \Gamma^2 c}{\lambda}\right)^i Q_2 \left(\frac{\Gamma \widetilde \Gamma}{\lambda^2}\right)^{i+1} = \frac{Q_1 Q_2 D\Gamma^2 \widetilde \Gamma}{\lambda^2} \left( \frac{\Gamma^3 \widetilde \Gamma^3 c}{\lambda^3}\right)^i.
\end{equation}
Using the final inequality in \eqref{eq:condition-I}: 
\begin{align*}
    \sum_{i=1}^{k-1} \mathfrak E_i^{(k)} \leq \frac{Q_1 Q_2 D\Gamma^2 \widetilde \Gamma}{\lambda^2} \sum_{i=1}^{\infty} \left( \frac{\Gamma^3\widetilde\Gamma^3 c}{\lambda^3}\right)^i = \frac{Q_1 Q_2D \Gamma^2 \widetilde \Gamma}{\lambda^2} \frac{1}{1-\frac{\Gamma^3 \widetilde \Gamma^3 c}{\lambda^3}} \frac{\Gamma^3 \widetilde \Gamma^3 c}{\lambda^3} = \frac{Q_1 Q_2 D \Gamma^5 \widetilde \Gamma^4 c}{\lambda^2(\lambda^3 - \Gamma^3 \widetilde \Gamma^3 c)}.
\end{align*}
\end{proof}

\subsection{Estimates for  \( \mathfrak E_i^{(k)} \) with Quasi-Hyperbolicity, Condition (II)}\label{subsubsec:var-pf-II}

We assume now that $\xi_0$ is quasi-hyperbolic up to time $k$ and satisfies  \eqref{eq:condition-II}. 
\begin{lemma}\label{lem:var-pf-1II}
    \begin{equation}\label{eq:var-pf-1II}
        \mathfrak E_0^{(k)} \leq \|D^2 \Phi_{\xi_0}(e^{(1)},\cdot)\| + \widetilde Q_{3}\frac{c}{\tilde c}.
    \end{equation}
\end{lemma}

\begin{proof}
    We note that by \eqref{eq:second-derivative-norm}, \eqref{eq:ex-platypus}, \eqref{eq:constants-basic2-sing1}, and \eqref{eq:cvgce-k-quasi-hyp-sac-1}:
    \begin{align*}
    \|[\partial_\varsigma (D\Phi_{\xi_0})]e^{(k)}\| &\leq \|[\partial_\varsigma (D\Phi_{\xi_0})] e^{(1)}\| + \|\partial_\varsigma(D\Phi_{\xi_0})\|\|e^{(k)} - e^{(1)}\| \\
    &\leq \|D^2 \Phi_{\xi_0}(e^{(1)},\cdot)\| + \|D^2 \Phi_{\xi_0}\| \|e^{(k)} - e^{(1)}\| \\
    &\leq \|D^2 \Phi_{\xi_0}(e^{(1)},\cdot)\| + \frac{\widetilde Q_1D\Gamma c}{\tilde c}.
    \end{align*}
\end{proof}

\begin{lemma}\label{lem:var-pf-2II}
    \begin{equation}\label{eq:var-pf-2II}
        \sum_{i=1}^{k-1} \mathfrak E_i^{(k)} \leq 
        \widetilde Q_{4 }\frac{c}{\tilde c} 
    \end{equation}
\end{lemma}

\begin{proof}
    We observe that $\|[\partial_\varsigma (D\Phi_{\xi_i})] e^{(k)}_i\| \leq \|\partial_\varsigma (D\Phi_{\xi_i})\|\|e^{(k)}_i\|$. By \eqref{eq:constants-basic2-sing1} and \eqref{eq:second-derivative-norm}, we get: 
    \begin{equation}\label{eq:partial-DPhi-typeII}
\|\partial_{\varsigma}(D\Phi_{\xi_{i}})\| \leq \|D^2\Phi_{\xi_i}\| \leq D\Gamma\widetilde\Gamma^i.
\end{equation}
From \eqref{eq:cvgce-k-quasi-hyp-sac-2} and \eqref{eq:partial-DPhi-typeII}, we get:
\begin{equation}\label{eq:dphi-deriv-bounds-typeII}
\|[\partial_\varsigma (D\Phi_{\xi_i})] e_i^{(k)}\| \leq \widetilde Q_1 \left( \frac{\Gamma c}{\tilde c}\right)^i D\Gamma \widetilde \Gamma^i = \widetilde Q_1 D \Gamma \left( \frac{\Gamma\widetilde\Gamma c}{\tilde c}\right)^i.
\end{equation}
By \eqref{eq:cvgce-k-quasi-hyp-sac-3} and \eqref{eq:dphi-deriv-bounds-typeII}, we get:
\begin{equation}\label{eq:i geq 1 typeII}
\mathfrak E_i^{(k)} = \frac{\|[\partial_\varsigma(D\Phi_{\xi_i})]e^{(k)}_i\|\|e^{(k)}_{i+1}\|}{|\det(D\Phi_{\xi_0}^{i+1})|} \leq \widetilde Q_1 D\Gamma \left( \frac{\Gamma\widetilde\Gamma c}{\tilde c}\right)^i \widetilde Q_2 \left( \frac{\Gamma}{\lambda^2\tilde c}\right)^{i+1} = \frac{\widetilde Q_1 \widetilde Q_2 D \Gamma^2}{\lambda^2\tilde c} \left( \frac{\Gamma^2 \widetilde \Gamma c}{\lambda^2 \tilde c^2}\right)^i.
\end{equation}
Summing up these terms gives us: 
\begin{align*}
    \sum_{i=1}^{k-1} \mathfrak E_i^{(k)} \leq \frac{\widetilde Q_1 \widetilde Q_2 D \Gamma^2}{\lambda^2\tilde c} \sum_{i=1}^\infty \left( \frac{\Gamma^2 \widetilde \Gamma c}{\lambda^2 \tilde c^2}\right)^i = \frac{\widetilde Q_1 \widetilde Q_2 D \Gamma^2}{\lambda^2\tilde c} \frac{1}{1-\frac{\Gamma^2 \widetilde \Gamma c}{\lambda^2 \tilde c^2}} \frac{\Gamma^2 \widetilde \Gamma c}{\lambda^2 \tilde c^2} = \frac{\widetilde Q_1 \widetilde Q_2 D \Gamma^4 \widetilde\Gamma c}{\lambda^2\tilde c(\lambda^2 \tilde c^2 - \Gamma^2 \widetilde \Gamma c)}.
\end{align*}
\end{proof}

\subsection{Estimates for   $\mathfrak F_i^{(k)}$}\label{sec:Fik}

We now estimate the third term of \eqref{eq-apriori}. Unlike the estimates for the $\mathfrak E_{i}^{(k)}$ terms, the $\mathfrak F_{i}^{(k)}$ terms do not require separate assumptions depending on whether $\xi_{0}$ satisfies conditions (I) or (II). 
\begin{lemma}\label{lem:var-pf-3I}
    \begin{equation}\label{eq:var-pf-3I}
        \frac{\|e^{(k)}_k\|^2}{\|f^{(k)}_k\|^2} \sum_{i=0}^{k-1} \mathfrak F_{i}^{k-1} \leq  Qc.
    \end{equation}
\end{lemma}

\begin{proof}
    First of all, we observe that
    \begin{equation}\label{eq:fbound-pf-1}
        |\det(D\Phi_{\xi_0}^{i+1})| = \frac{|\det(D\Phi_{\xi_0}^k)|}{|\det(D\Phi_{\xi_{i+1}}^{k-i-1})|} = \frac{\|e^{(k)}_k\|\|D\Phi_{\xi_0}^k\|}{|\det(D\Phi_{\xi_0}^{k-i-1})|},
    \end{equation}
    and therefore, using also the facts that $\|f^{(k)}_{i+1}\| \leq \|D\Phi^{i+1}_{\xi_0}\|$ and $\|\partial_\varsigma(D\Phi_{\xi_i})\| \leq \|D^2 \Phi_{\xi_i}\|$, 
    \begin{equation}\label{eq:fbound-pf-2}
    \begin{aligned}
        \frac{\|[\partial_\varsigma(D\Phi_{\xi_i})]f^{(k)}_i\|\|f^{(k)}_{i+1}\|}{|\det(D\Phi_{\xi_0}^{i+1})|} &\leq \frac{\|D^2\Phi_{\xi_i}\| \|D\Phi^i_{\xi_0}\|\|D\Phi^{i+1}_{\xi_0}\|}{|\det(D\Phi_{\xi_0}^{i+1})|} \\
        &= \frac{\|D^2\Phi_{\xi_i}\|\|D\Phi_{\xi_0}^i\|\|D\Phi_{\xi_0}^{i+1}\||\det(D\Phi_{\xi_{i+1}}^{k-i-1})|}{\|e^{(k)}_k\|\|D\Phi_{\xi_0}^k\|}.
        \end{aligned}
    \end{equation}
Therefore, applying \eqref{eq:constants-basic-sing1} and \eqref{eq:constants-basic2-sing1} from Definition \ref{def:nonsingular-QHv2-I} to \eqref{eq:fbound-pf-2}, we obtain:
    \begin{equation}\label{eq:fbound-pf-3I}
    \begin{aligned} 
     \mathfrak F_{i}^{(k)} &= \frac{\|[\partial_\varsigma(D\Phi_{\xi_i})]f^{(k)}_i\|\|f^{(k)}_{i+1}\|}{|\det(D\Phi_{\xi_0}^{i+1})|} \leq \frac{D^3\Gamma^{2i+2}\widetilde\Gamma^i b^{k-i-1}}{\|e^{(k)}_k\|C\lambda^k} = \frac{D^3\Gamma^2 b^{k-1}}{\|e^{(k)}_k\|C\lambda^k} \left( \frac{\Gamma^2\widetilde\Gamma}{b}\right)^i,
     \end{aligned} 
    \end{equation}
    and therefore, 
    \begin{equation}\label{eq:fbound-pf-4I}
    \frac{\|e^{(k)}_k\|^2}{\|f^{(k)}_k\|^2} \sum_{i=0}^{k-1} \mathfrak F_{i}^{k-1} \leq \frac{\|e^{(k)}_k\| D^3 \Gamma^2 b^{k-1}}{\|f^{(k)}_k\|^2 C\lambda^k} \sum_{i=0}^{k-1} \left( \frac{\Gamma^2\widetilde\Gamma}{b}\right)^{i}
    \end{equation}
    Estimating this sum, we find: 
    \begin{equation}\label{eq:fbound-pf-5I}
    \sum_{i=0}^{k-1} \left( \frac{\Gamma^2\widetilde\Gamma}{b}\right)^i = \frac{\left( \frac{\Gamma^2\widetilde\Gamma}{b}\right)^k - 1}{ \frac{\Gamma^2\widetilde\Gamma}{b}-1} \leq \frac{\left( \frac{\Gamma^2\widetilde\Gamma}{b}\right)^k}{ \frac{\Gamma^2\widetilde\Gamma}{b}-1} = \frac{b}{\Gamma^2\widetilde\Gamma - b} \left( \frac{\Gamma^2\widetilde\Gamma}{b}\right)^k
    \end{equation}
(note the inequality above is where we use the final assumption in \eqref{eq:constants-basic-x} that $\Gamma^2 \widetilde \Gamma > b$). Plugging this into \eqref{eq:fbound-pf-4I}, and also recalling that $\|e^{(k)}_k\|/\|f^{(k)}_k\| < Bc^k$ and $\|f^{(k)}_k\| \geq C\lambda^k$, we have: 
    \[
    \begin{aligned}
    \frac{\|e^{(k)}_k\|^2}{\|f^{(k)}_k\|^2} \sum_{i=0}^{k-1} \mathfrak F_{i}^{k-1} &\leq \frac{\|e^{(k)}_k\| D^3\Gamma^2 b^{k}}{\|f^{(k)}_k\|^2 C\lambda^k(\Gamma^2\widetilde\Gamma - b)} \left( \frac{\Gamma^2\widetilde\Gamma}{b}\right)^k \\
    &\leq \frac{BD^3\Gamma^2(bc)^k}{C^2\lambda^{2k}(\Gamma^2\widetilde\Gamma -b)} \left( \frac{\Gamma^2\widetilde\Gamma}{b}\right)^k \\ 
    &= \frac{BD^3\Gamma^2}{C^2(\Gamma^2\widetilde\Gamma - b)} \left( \frac{\Gamma^2 \widetilde\Gamma c}{\lambda^2}\right)^k.
\end{aligned}
    \]
%\todelete{Finally, recall that by \eqref{eq:constants-basic-x}, \( \Gamma^{2}\widetilde\Gamma/\lambda^{2}< 1\),  by \eqref{eq:constants-basic-sing1},  \( c <1 \) and so  \( (\Gamma^2 \widetilde\Gamma c/\lambda^2)^k < \Gamma^2 \widetilde\Gamma c/\lambda^2\). Therefore,}

Finally, since $\widetilde\Gamma \geq 1$ and $\Gamma > \lambda$, and since $\tilde c < 1$, both \eqref{eq:condition-I} and \eqref{eq:condition-II} imply that $\Gamma^2\widetilde\Gamma c/\lambda^2 \leq 1$. Therefore,
\[
     \frac{\|e^{(k)}_k\|^2}{\|f^{(k)}_k\|^2} \sum_{i=0}^{k-1} \mathfrak F_{i}^{k-1} 
  \leq  \frac{BD^3\Gamma^2}{C^2(\Gamma^2\widetilde\Gamma - b)} \left( \frac{\Gamma^2 \widetilde\Gamma c}{\lambda^2}\right)^k
  \leq 
  \frac{BD^3\Gamma^2}{C^2(\Gamma^2\widetilde\Gamma - b)} \frac{\Gamma^2 \widetilde\Gamma}{\lambda^2} c
    \]
 which gives the statement in the Lemma.   
\end{proof}

\subsection{Proof of first inequality in Theorem \ref{thm:main-phase}}
\label{sec:1stineq}

Suppose condition  \eqref{eq:condition-I} is satisfied. 
Substituting the bounds in  Lemmas \ref{lem:var-pf-1I}, \ref{lem:var-pf-2I}, and \ref{lem:var-pf-3I} into \eqref{eq:aposteriori} we get 
\[
  \|Df^{(k)}\| \leq K_{1} ( \|D^2 \Phi_{\xi_0}(e^{(1)},\cdot)\| + Q_{3}c + Q_{4}c + Q c)
\]
which gives the first inequality in Theorem \ref{thm:main-phase} with \( K_{2}= K_{1}(Q_{3}+Q_{4}+Q_{5})\). 
Similarly, supposing condition \eqref{eq:condition-II} is satisfied, substituting the bounds in  \ref{lem:var-pf-1II}, \ref{lem:var-pf-2II}, and \ref{lem:var-pf-3I}, into \eqref{eq:aposteriori} we get 
\[
  \|Df^{(k)}\| \leq K_{1} ( \|D^2 \Phi_{\xi_0}(e^{(1)},\cdot)\| + \widetilde Q_{3}c + \widetilde Q_{4}c + Q c)
\]
which gives the first inequality in Theorem \ref{thm:main-phase} with \( K_{2}= K_{1}(\widetilde Q_{3}+ \widetilde Q_{4}+Q_{5})\). 

\printbibliography

\medskip
\textbf{Stefano Luzzatto}
\\
Abdus Salam International Centre for Theoretical Physics (ICTP), Trieste, Italy. 
\\
\url{https://www.stefanoluzzatto.net}
\\
\href{mailto:luzzatto@ictp.it}{\texttt{luzzatto@ictp.it}}

\medskip
\textbf{Dominic Veconi}
\\
Abdus Salam International Centre for Theoretical Physics (ICTP), Trieste, Italy. 
\\
\emph{Current Affiliation:} 
Wake Forest University, Winston-Salem, NC, USA
\\
\url{https://dominic.veconi.com}\\
\href{mailto:veconid@wfu.edu}{\texttt{veconid@wfu.edu}}

\medskip
\textbf{Khadim War}
\\
Instituto de Matematica Pura e Aplicada (IMPA), Rio de Janeiro, Brazil
\\
\url{https://sites.google.com/view/khadim-war/home}
\\
\href{mailto:warkhadim@gmail.com}{\texttt{warkhadim@gmail.com}}

\end{document}